%% file: CutoffC.tex
\documentclass[11pt,a4paper]{article}
\usepackage[utf8]{inputenc}
\usepackage[T1]{fontenc}
\usepackage{amsmath}
\usepackage{amsfonts}
\usepackage{amssymb}
\usepackage{dsfont}
\usepackage{fullpage}
\usepackage{graphicx}
\usepackage{hyperref}
\usepackage{color}
\usepackage{mathrsfs}
\usepackage{enumitem}
\usepackage{listings}
\usepackage{amsthm}
\usepackage{listings}
\usepackage[title]{appendix}
\usepackage{comment}
\numberwithin{equation}{section}
\numberwithin{figure}{section}
\theoremstyle{plain}
\newtheorem{thm}{\protect\theoremname}[section]
  \theoremstyle{remark}
  \newtheorem{rem}[thm]{\protect\remarkname}
  \theoremstyle{definition}
  
  \theoremstyle{plain}
  \newtheorem{lem}[thm]{\protect\lemmaname}
  \theoremstyle{plain}
  \newtheorem{prop}[thm]{\protect\propositionname}
  \theoremstyle{remark}
  
  \theoremstyle{remark}
  
  \theoremstyle{remark}
  
    \theoremstyle{plain}
  \newtheorem{cor}[thm]{\protect\corollaryname}

 \providecommand{\claimname}{Claim}
  \providecommand{\definitionname}{Definition}
  \providecommand{\lemmaname}{Lemma}
  \providecommand{\observationname}{Observation}
  \providecommand{\propositionname}{Proposition}
  \providecommand{\remarkname}{Remark}
  \providecommand{\examplename}{Example}
\providecommand{\theoremname}{Theorem}
\providecommand{\corollaryname}{Corollary}

\newcommand{\Z}{\mathbb{Z}}

\newcommand{\R}{\mathbb{R}}
\newcommand{\N}{\mathbb{N}}
\newcommand{\E}{\mathbb{E}}
\renewcommand{\P}{\mathbb{P}}
\newcommand{\h}{\mathcal{H}}
\newcommand{\s}{\sigma}
\newcommand{\e}{\eta}
\newcommand{\z}{\zeta}

\newcommand{\el}{\ell}

\newcommand{\Cov}{\textup{Cov}}
\newcommand{\Var}{\textup{Var}}
\newcommand{\one}{\mathbf{1}}
\newcommand{\oneL}{{\mathbf{1}_L}}
\def\restriction#1#2{\mathchoice
              {\setbox1\hbox{${\displaystyle #1}_{\scriptstyle #2}$}
              \restrictionaux{#1}{#2}}
              {\setbox1\hbox{${\textstyle #1}_{\scriptstyle #2}$}
              \restrictionaux{#1}{#2}}
              {\setbox1\hbox{${\scriptstyle #1}_{\scriptscriptstyle #2}$}
              \restrictionaux{#1}{#2}}
              {\setbox1\hbox{${\scriptscriptstyle #1}_{\scriptscriptstyle #2}$}
              \restrictionaux{#1}{#2}}}
\def\restrictionaux#1#2{{#1\,\smash{\vrule height .8\ht1 depth .85\dp1}}_{\,#2}} 

\usepackage{pgfplots}
\usetikzlibrary{calc}
\usetikzlibrary{patterns}
\definecolor{myblue}{rgb}{0,0.5,1}
\definecolor{mypink}{rgb}{1,0,1}
\definecolor{mygreen}{rgb}{0,0.6,0}
 \pgfdeclarepatternformonly{north west spaced lines}{\pgfqpoint{-1pt}{-1pt}}{\pgfqpoint{10pt}{10pt}}{\pgfqpoint{9pt}{9pt}}%
{
  \pgfsetlinewidth{0.4pt}
  \pgfpathmoveto{\pgfqpoint{0pt}{10pt}}
  \pgfpathlineto{\pgfqpoint{10.1pt}{-0.1pt}}
  \pgfusepath{stroke}
}

\begin{document}

\begin{LARGE}
\begin{center}
\textbf{Cutoff for the Fredrickson-Andersen one spin facilitated model}
\end{center}
\end{LARGE}

\begin{Large}
\begin{center}
Anatole Ertul \footnote{Institut Camille Jordan, Universit\'e Lyon 1, France.
    \textit{ertul@math.univ-lyon1.fr}}
\end{center}
\end{Large}

\begin{abstract}
The Fredrickson-Andersen one spin facilitated model belongs to the class of Kinetically Constrained Spin Models. It is a non attractive process with positive spectral gap. In this paper we give a precise result on the relaxation for this process on an interval $[1,L]$ starting from any initial configuration. A consequence of this result is that this process exhibits cutoff at time $L/(2v)$ with window $O(\sqrt{L})$ for a certain positive constant $v$. The key ingredient is the study of the evolution of the leftmost empty site in a filled infinite half-line called the front. In the process of the proof, we improve recent results about the front motion by showing that it evolves at speed $v$ according to a uniform central limit theorem.
\end{abstract}

\section{Introduction}

The Fredrickson--Andersen one spin facilitated model (FA-1f) model is an interacting particle system that belongs to the class of kinetically constrained spin models (KCSM). These models are Markov processes that were first introduced in the 1980's by physicists to model liquid-glass transitions (see \cite{FA84} and \cite{JE}) and that have some analogy with the Glauber dynamics of the Ising model. In a KCSM on a graph $G$, each vertex can either have a particle or be empty. At rate 1, each site tries to update according to a Bernoulli measure, but does so only if a local constraint is satisfied. On the graph $\Z$, the constraint can for example be that the site immediately to the right is empty, in which case we get the \emph{East model}. Another possible constraint is to have either of the adjacent neighboring sites empty, which defines the FA-1f model.\\

Since the 2000's, the dynamics of KCSM both at and out of equilibrium have been thoroughly studied by the mathematical community. In 2002, Aldous and Diaconis \cite{AD02} proved that the East model has a positive spectral gap. This result has been generalized to a large class of models in \cite{CMRT08}. Starting out of equilibrium studies are made more difficult by the lack of attractiveness of the processes, preventing usual monotonicity and coupling arguments. Some results can be found for example in \cite{CMST} for the East model and \cite{MV, BCM13} for FA-1f at low density. For the East model and the FA-1f model on $\Z$, an interesting topic of study is the motion of the front, i.e. the rightmost empty site of a configuration filled on an infinite half-line. First, Blondel \cite{B13} showed that the front has a linear speed in the East model. Based on this work, Ganguly, Lubetzky and Martinelli showed in \cite{GLM} a central limit theorem for the front, which was followed by the proof that the East model exhibits cutoff thanks to a simple coupling argument. In \cite{BDT19}, it was shown that the front for the FA-1f process also behaves according to a CLT when the density is below a certain threshold. The aim of this paper is to use this last result to prove a cutoff for the FA-1f process. It will require much more work than for East because the main coupling argument \cite[Section 4.2]{GLM} no longer holds for FA-1f.\\

The cutoff phenomenon was first exhibited by Aldous and Diaconis in the context of card shuffling \cite{AD86}. It consists of a sharp drop of the total variation distance to the equilibrium measure of a Markov process (see \cite{LPW} for an introduction to this phenomenon). Since then, examples and counter-examples of cutoff have be shown for a wide variety of processes, but a universal criterion is missing. In 2004, Peres conjectured that a process exhibits cutoff if and only if it satisfies the \emph{product condition} $t_{rel} = o(t_{mix})$, where $t_{rel}$ is the inverse of the spectral gap and $t_{mix}$ is the total variation mixing time of the process. This condition turned out not to be sufficient in general but sufficient for a large class of processes (see \cite{BHP} and more recently \cite{S21}). Recently, several cutoff results have been shown for particle systems like the Ising model \cite{LS}, the Asymmetric Simple Exclusion Process \cite{LL} or a stratified random walk on the hypercube \cite{BP}.\\

In this paper, we will study in detail the relaxation of the FA-1f process depending on the initial configuration. We mainly rely on a result giving a bound on the mixing time for a configuration with many empty sites. The core of our study will thus be to see how quickly any initial configuration can create enough empty sites. For that, we study the big intervals that are initially filled with particles, and interpret the endpoint of those as fronts going inward at speed $v$. Thanks to another result referred to as "Zeros Lemma", we will be able to show that the configuration becomes suitable for relaxation after roughly the time it takes for the fronts of the biggest particle cluster to meet. Finally, our main result gives precise bounds on the mixing time of the process for \emph{any} initial configuration, which is a more complete result than just a cutoff statement.\\
  
First, we give the definition of the model and our main result in Section \ref{sec:2}, then an important coupling with a threshold contact process in Section \ref{sec:3} and results about relaxation to equilibrium in Section \ref{sec:4}. In Section \ref{sec:5}, we extend the central limit theorem proved in \cite{BDT19} in order to fit the actual proof of the main result, which appears in Section \ref{sec:6}.

\section{Definitions and main result}\label{sec:2}

\subsection{FA-1f processes}

We will encounter three types of FA-1f processes, depending on their state space. Fix $q \in [0,1]$. The FA-1f process of parameter $q$ on an interval $\Lambda \subset \Z$ is given by the generator:
$$\mathcal{L}f(\s) = \sum_{x\in \Lambda} r(x,\s)(f(\s^x) - f(\s)),$$
for any local function $f$ and $\s \in \{0,1\}^\Z$, where $\s^x$ is the configuration equal to $\s$ everywhere except at site $x$. The rate $r(x,\s)$ is given by:
$$r(x,\s) = (1-\s(x-1)\s(x+1))(q\s(x) + (1-q)(1-\s(x))).$$
$\s(x)=1$ is to be interpreted as the presence of a particle at $x$, and $\s(x)=0$ as an empty site. In words, every site makes a flip $0 \rightarrow 1$ at rate $1-q$ and $1 \rightarrow 0$ at rate $q$ but only if it satisfies the kinetic constraint $c_x(\s) := (1-\s(x-1)\s(x+1))$ is equal to $1$, that is if it has a least one empty neighbor. If $\Lambda$ has boundaries, we set the configuration to be fixed at $0$ at the (outer) boundaries.  We now fix some notations and conventions used throughout the paper.

\begin{itemize}
\item FA-1f processes on $\Z$ are generally denoted by the letter $\s$. Their state space is $\Omega := \{0,1\}^{\Z}$. We define $LO = \{\s \in \Omega \, | \, \exists X, \forall x<X, \s(x)=1\}$ the set of configurations equal to one on an infinite left-oriented half line.
\item FA-1f processes on $\Z_- := \{-1,-2,...\}$ are generally denoted by the letter $\e$. Their state space is $\Omega_- := \{\e \in \Omega \ |\ \e(0) = 0, \ \forall x \ge 0, \e(x) = 1 \}$. We define $LO_- = LO \cap \Omega_-$.
\item For a finite interval $\Lambda = [1,L]$, we define $\Omega_L = \{\s \in \Omega \ | \ \s(0) = \s(L+1) = 0, \ \forall x \notin [0,L+1], \s(x) = 1 \}$.
\end{itemize}

For $\s \in \{0,1\}^\Z$ and $x \in \Z$, we define the shifted configuration $\theta_x \s$ by: $$\forall y \in \Z, \ \theta_x \s(y) = \s(y+x).$$

For any $\s \in LO$, we define the front $X(\s)$ (or sometimes simply written $X$) as the leftmost zero in $\s$. Note that if the process is defined on an interval $\Lambda \ne \Z$, the front can be an outer boundary of $\Lambda$. When $(\s_t)_{t\ge0}$ is a process, we will write $X_t$ instead of $X(\s_t)$ when $\s_t$ clear from the context. We also denote by $\tilde{\s}$ the configuration seen from the front, namely $\theta_{X(\s)}\s$.\\

We also define $\mu^{\s}_t$ (resp.\ $\tilde{\mu}^{\s}_t$) the law of the FA-1f process (resp.\ seen from the front) at time $t$ starting from $\s$. Whether the process takes place in $\Z$ or $\Z_-$ is implicitly given by the configuration $\s$. We denote by $\mu^p_{\Lambda}$ the Bernoulli product law of parameter $p$ on $\{0,1\}^\Lambda$. The FA-1f process on $\Lambda$ with parameter $q$ is reversible with respect to $\mu^{1-q}_\Lambda$. In the following, we shall write $\mu$ instead of $\mu^{1-q}_\Lambda$ when $q$ and $\Lambda$ are clearly fixed.

Finally we define, for $I = [a,b] \subset \Lambda$ an interval and $\el \ge 10$ the set:
$$\h(I,\el) = \h(a,b,\el) := \{ \s \in \{0,1\}^\Lambda |\, \forall x\in [a,b-\el+1],  \exists y\in [x,x+\el-1], \ \s(y)=0  \}.$$ 
If $J = \bigcup_k I_k$ is a union of disjoint non adjacent intervals, then we define $\h(J,\el) := \bigcap_k \h(I_k,\el)$.
In words, in a configuration in $\h(I,\el)$, every site is at distance at most $\el -1$ to either the right boundary of $I$ or to an empty site to its right. In such a configuration, two consecutive empty sites are within distance $\el$.

\subsection{Main result}

We denote $\Omega^\delta_L = \{ \s \in \Omega_L \ | \ \s \in \h(1,L,\delta L) \}$. We also define, for $\s \in \Omega_L$, $$B(\s) := \max \{ h \ge 0 \ | \ \exists x \in [0,L-h], \ \restriction{\s}{[x+1,x+h]} \equiv 1 \}$$ the size of the largest component of occupied sites in $\s$. The most precise result about the relaxation of the FA-1f process in a finite interval proved in this paper is the following. Recall the usual notations for real numbers: $a \wedge b = \max(a,b), \, a \vee b = \min(a,b)$.
\begin{thm}\label{mainthm}
There exists $\bar{q} < 1$ such that for every $q > \bar{q}$, $\delta \in (0,1)$ and $\varepsilon > 0$, the following holds. There exist three constants $a = a(\varepsilon,q) > 0$ and  $0 < \underline{v} = \underline{v}(q) < v = v(q) $ such that, if we first define the following three times for any $\s \in \Omega_L$:
\begin{align*}
t_1(\s) &= \left( \dfrac{B(\s)}{\underline{v}}\right) \vee \left( \dfrac{(\log L)^9}{\underline{v}}\right), \\
t_2(\s) &= \dfrac{B(\s)}{2v} + \dfrac{a}{\underline{v}\delta}\sqrt{B(\s)}, \\
t_3(\s) &= \dfrac{B(\s)}{2v} - \dfrac{a}{v} \sqrt{B(\s)},\\
\end{align*}

\noindent then:

 \begin{equation}\label{eq:mainthm1}
\underset{L \rightarrow +\infty}{\limsup} \ \underset{\s \in \Omega^\delta_L}{\sup} ||\mu^\s_{t_1(\s)} - \mu ||_{TV} = 0,
\end{equation}

\begin{equation}\label{eq:mainthm2}
\underset{L \rightarrow +\infty}{\limsup} \ \underset{\s \in (\Omega^\delta_L)^c}{\sup} ||\mu^\s_{t_2(\s)} - \mu ||_{TV} \le \dfrac{\varepsilon}{\delta^2}.
\end{equation}

Moreover, for any function $\Phi$ such that $\Phi(L) \underset{L \rightarrow +\infty}{\longrightarrow} + \infty$, 
\begin{equation}\label{eq:mainthm3}
\underset{L \rightarrow +\infty}{\liminf} \ \underset{\s \in \h(1,L,\Phi(L))^c}{\inf} ||\mu^\s_{t_3(\s)} - \mu ||_{TV} \ge 1 - \varepsilon.
\end{equation}

\end{thm}

Let us give some comments about this statement. First, the two constants $\underline{v}$ and $v$ have the following interpretations. The constant $\underline{v}$ is the speed of the infection propagation for a threshold contact process defined in the following Section. The constant $v$ corresponds to the speed of the front for the FA-1f process on $\Z$ starting with an infinitely filled half-line, that is with initial configuration in $LO$. Both $\underline{v}$ and $v$ depend on $q$ that we chose here to be greater to a certain threshold $\bar{q}$ necessary for the threshold contact process to survive. As we will see, the value $\bar{q}$ is roughly equal to $0.76$ \cite{BFM}.
The first two equations \eqref{eq:mainthm1} and \eqref{eq:mainthm2} of Theorem \ref{mainthm} give an upper bound on the time at which the process is well mixed. The first one works nicely for initial configurations that already have enough empty sites. Indeed, in this case $B(\s)$ is smaller than $\delta L$, making up for the loss in the constant $1/(2\underline{v})$ instead of the expected $1/(2v)$. The time $t_1$ has to be bounded from below by $(\log L)^9$ for technical reasons we will see later. This bound is most likely not optimal but still offers a precise upper bound on the mixing time. The second equation handles the configurations that have macroscopic components of occupied sites and has the optimal leading behaviour $\frac{B(\s)}{2v}$ one can expect as shown in the last equation \eqref{eq:mainthm3}. The last equation gives a lower bound on the mixing time for configurations with at least an interval of size $\Phi(L)$ occupied. Since it is only required that $\Phi(L)$ goes to infinity, this hypothesis is not really restrictive. It is however necessary for our argument (Theorem \ref{TCL}) to work.
As a consequence of this theorem, we can conclude that the FA-1f process exhibits a cutoff.

\begin{thm}\label{thm:cutoff}
Let $d(t) = \underset{\s \in \Omega_\Lambda}{\sup} || \mu_t^\s - \mu ||_{TV}$. Then for all $q > \bar{q}$ and $\varepsilon > 0$, there exist $\alpha(\varepsilon,q), \beta(\varepsilon, q) > 0$ independent of $L$, such that for $L$ large enough:
\begin{align*}
d\left(\frac{L}{2v} - \alpha(\varepsilon)\sqrt{L}\right) &\ge 1 - \varepsilon,\\
d\left(\frac{L}{2v} + \beta(\varepsilon)\sqrt{L}\right) &\le \varepsilon.
\end{align*}
\end{thm}

\section{Coupling with a contact process and consequences}\label{sec:3}

\subsection{Graphical construction}

Although we can easily define the FA-1f process through its generator, we also provide a more convenient construction called the graphical construction. For each site $x \in \Z$, define a Poisson process $T^x$ of parameter $1$ and an infinite sequence of Bernoulli random variables $(\beta^x_n)_{n \ge 1}$ of parameter $1-q$, all of these variables being independent. Given an initial configuration $\s_0$, at each increment $t$ of a Poisson process, say the $n^{\textup{th}}$ increment at site $x$, we check if the constraint $c_x(\s)$ is satisfied, that is if the site $x$ has at least one empty neighbor. If it does, then we set $\s_{t}(x) = \beta^x_n$. Otherwise, $\s_{t}(x)$ is unchanged.\\
\indent From this construction, we can define the \emph{standard coupling} which simply consists in taking the same set of random variables $(T^x)_{x \in \Z}$ and $(\beta^x_n)_{n \ge 1}$ for different initial configurations. Note that this coupling is not monotone: one can have $\s_0 \le \s'_0$ and $\s_t \not\le \s'_t$ even if $(\s_t)_{t\ge 0}$ and $(\s'_t)_{t \ge 0}$ follow the standard coupling (see Figure \ref{fig:monot}). This is an important reason why this dynamics can be difficult to study.

\begin{figure}[!h] 
\centering
\includegraphics[scale=0.4]{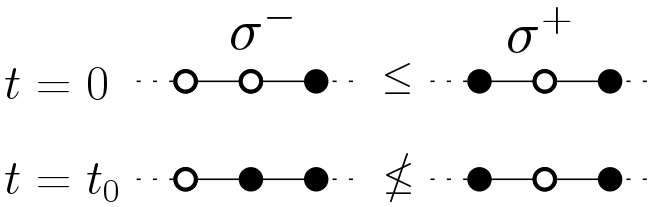} 
\caption{\label{fig:monot} An example of non monotonicity via the standard coupling. The middle site could update in $\s_-$ but not in $\s_+$.}
\end{figure}

Using the graphical construction, it is standard to show that there is a finite speed of propagation:

\begin{prop}[{\cite[Lemma 2.5]{BDT19}}] \label{eventF}
Let $(\s_t)_{t\ge0}$ be a FA-1f process. For $t \ge 0$ and two sites $x < y$, we define the events:\\
$F(x,y,t) = \{ \text{there is a sequence of successive rings linking x and y in a time interval of length t} \},$
$\tilde{F}(x,y,t) = \{\exists z \in [x,y] \ \textup{s.t.} \ F(x,z,t)\cap F(y,z,t)\}.$\\
Then there exists a constant $\bar{v} > 0$ such that if $|x-y| \ge \bar{v}t$, then
$$\P(F(x,y,t)) \le \P(\tilde{F}(x,y,t)) \le e^{-|x-y|}.$$
\end{prop}

This implies a maximum speed for the front:

\begin{cor}\label{vitessemax}
For all $\s \in LO$ or $\e \in LO_-$, and $c \ge \bar{v}$, $\P(X_t -X_0 < -ct) \le e^{-ct}$.
\end{cor}

\subsection{Contact process}

In \cite{H74}, Harris introduced the contact process on $\Z$. It is a process on $\{0,1\}^\Z$ that allows sites to flip from $0$ (infected site) to $1$ at rate $1$, and from $1$ to $0$ at a rate depending on the amount of empty adjacent sites. Based on this process, we introduce the same \emph{threshold contact process} as in \cite{BDT19}, denoted here by $(\z_t)_{t\ge 0}$ and given by the following generator:
$$\forall f \ \text{local}, \, \forall \z \in \Omega, \, \mathcal{L'}f(\z) = \sum_{x \in \Z} r'(x,\z)(f(\z^x)-f(\z)),$$
where $r(x,\z) = (1-\z(x-1)\z(x+1))q\z(x) + (1-q)(1-\z(x))$. Note that we took here the convention that $0$ is an infected site, which is not the most common one.\\

This process differs from both FA-1f and Harris' process in its constraint: a site is free to flip from $0$ to $1$ at rate $p$ but has to have at least one empty neighbor to flip from $1$ to $0$ at rate $q$.  This threshold contact process is the one well suited to a coupling with FA-1f as we will explain later. Let $\bar{q}$ be the same parameter than in \cite[Appendix B]{BDT19} such that whenever $q \ge \bar{q}$, then the threshold contact process starting from a single infected site had a positive probability of surviving and will create infected sites in an interval scaling linearly with time.\\

Let $\e_0 \in LO_-$ and $(\e_t)_{t \ge 0}$ be a FA-1f process starting from $\e_0$. Let $\z_0 = \delta_{X_0}$ where $X_0$ is the front of $\e_0$ and $\delta_{X_0}$ is the configuration (on $\Z$) equal to $0$ in site $X_0$ and equal to $1$ everywhere else. Define $(\z_t)_{t\ge0}$ the contact  process starting from configuration $\z_0$, evolving with respect to the standard coupling with $(\e_t)_{t \ge 0}$. Then $\forall t \ge 0, \forall x \in \Z_-, \e_t(x) \le \z_t(x)$. Indeed, this inequality holds for $t = 0$. Then, for each ring of a site $x \in \Z_- \setminus \{0\}$, assuming $\forall y \in \Z_-, \e_{t^-}(y) \le \z_{t^-}(y)$, there are three cases:
\begin{itemize}
\item The kinetic constraint at $x$ is not satisfied in $\e_{t^-}$, in which case it is not satisfied in $\z_{t^-}$ either. In this case, the value $\e_t(x)$ does not change and the only possible change for $\z_t(x)$ is $0 \rightarrow 1$, which preserves the order.
\item The kinetic constraint at $x$ is satisfied in $\e_{t^-}$ but not in $\z_{t^-}$. In this case, a $0 \rightarrow 1$ flip is possible for both configurations, and a $1 \rightarrow 0$ flip is only possible for $\e_t(x)$, which does not break the inequality.
\item The constraint is satisfied for both configurations, and then by construction of the coupling, $\e_{t^-}(x) = \z_{t^-}(x)$, which again does not change the inequality.
\end{itemize}
Note also that the behavior of $\z$ on $\Z_+$ influences the behavior on $\Z_-$ only through the $x=0$ site. Since $\forall t, \e_t(0) = 0$, we clearly have $\e_t(0) \le \z_t(0)$.

Thanks to this observation, and to the known results on the contact process \cite[Appendix B]{BDT19}, we can now state a first important lemma which provides a minimal speed for the front.

\begin{lem} \label{vitessemin}
Let $q>\bar{q}$. There exists $\underline{v} > 0$ and $A,B >0$ such that for every $\e \in LO_-$, and $t \ge 0$,
$$\P(X_t - X_0 > -\underline{v}t) \le A e^{-Bt}$$
\end{lem}

The proof of the result is based on the comparison explained above. Hence, it is identical to the proof of Corollary 4.2 in \cite{BDT19}.

\subsection{Zeros Lemmas}

Next, we state a crucial lemma based on the coupling explained above and on results for surviving contact processes. It is the same result as \cite[Corollary 4.3]{BDT19}, only extended to the FA-1f process on $\Z_-$ and $[1,L]$, and it follows from the same coupling argument.

\begin{lem}\label{lem:contact}
Let $q>\bar{q}$. There exist $c_1,c_2 > 0$ such that for any $\e_0 \in \Omega_-$ and $x \le 0$ such that $\e_0(x) = 0$,
$$ \forall t \ge 0, \ \P(\e_t \notin \h(x-\underline{v}t,(x+\underline{v}t) \wedge 0,\el) \le c_1 t \exp(-c_2(t \wedge \el)). $$
For $L>0$, if $\s_0 \in \Omega_L$ and $x \in [0,L+1]$ is such that $\s_0(x)=0$, then similarly,
 $$ \forall t \ge 0, \ \P(\s_t \notin \h((x-\underline{v}t) \vee 0 ,(x+\underline{v}t) \wedge (L+1),\el) \le c_1 t \exp(-c_2(t \wedge \el)). $$
\end{lem}

This last lemma can be used with $x$ being the front of a configuration $\e \in LO_-$ (or its boundary). This idea leads to a result we will now refer to as the "Zeros Lemma". We give here two versions of this result. The first one matches Lemma 4.4 in \cite{BDT19} while the second one is more specific to our needs later on.

\begin{lem}[Zeros Lemma]\label{zeros1}
Let $q > \bar{q}$. Let $s,\el,M,L > 0$ and $\e \in LO_-$.
\begin{itemize}
\item If $L+M \le 2\underline{v}s$, there exists $c >0$ depending only on $q$ such that:
$$ \P_\e\left(\tilde{\e}_s \notin \h(L,(L+M) \wedge (-X_s),\el)\right) \le (L+M)^2 \exp(-c(L \wedge \el)). $$
\item If $L+M > 2\underline{v}s$ and $\tilde{\e_0} \in \h(0,(L+M) \wedge (-X_0),2\underline{v}s)$, then there exists $c>0$ such that:
$$ \P_\e\left( \tilde{\e}_s \notin \h(L,(L+M) \wedge -X_s,\el)\right) \le \dfrac{s^2}{L}\exp(-c(L \wedge \el)) + Ms \exp(-c(s \wedge \el)).$$
\end{itemize}
\end{lem} 

\begin{proof}
It is identical to the proof in \cite{BDT19}, just taking into account the boundary.
\end{proof}

\begin{lem}[Zeros Lemma II] \label{zeros2}
Let $q > \bar{q}$. Let $s,\el > 1$ and $\e \in LO_-$. Assume $\e \in \h(X_0,y,\el)$ for some $y \in [X_0,0]$, where $X_0$ denotes the leftmost zero of $\e$.
Then, there exists $C >0$ depending only on $q$ such that if $2\underline{v}s \ge \el$, $$ P_{\e}\left(\e_s \notin \h(X_s,y,\el)\right) \le C(1 + |X_0|) \frac{s^2}{\el}\exp({-c (s \wedge \el)}).$$
\end{lem}

\begin{rem}
From now on and all throughout the proofs of this paper, the letters $c, C, c', C'$ always denote generic positive constants that may differ from line to line. Consequently, any equation such as $$`` \ \P(A_t) \le Ce^{-ct} \ "$$ should be understood as
$$`` \ \exists C,c > 0, \  \P(A_t) \le Ce^{-ct} \ ".$$
The dependency of these constants on the various parameters appearing in this paper should be explicit or clear from the context.
\end{rem}

\begin{proof}
In the same way as the proof of \cite[Lemma 4.4]{BDT19}, we use Lemma \ref{lem:contact} at intermediate times.
Let us fix $s$ and define:
\begin{align*}
\Delta &= \dfrac{\el}{2(\bar{v}-\underline{v})}\wedge s,\\
n &= \left\lceil \dfrac{(s-\Delta)(\bar{v}-\underline{v})}{2\underline{v}\Delta} \right\rceil, \\
\Delta' &= \dfrac{s-\Delta}{n} \ \text{if} \ n > 0, \\
s_i &= i\Delta' \ \textup{for} \ i \in [0,n].
\end{align*}
For $i \in [0,n]$, by Lemma \ref{lem:contact} and Markov property applied at time $s_i$, there exist $c,C > 0$ such that:
\begin{equation}\label{eq:proofzeros1}
\P_\e \left(\e_s \notin \h(X_{s_i} - \underline{v}(s-s_i), 0 \wedge (X_{s_i} + \underline{v}(s-s_i)), \el/2) \right) \le C(s-s_i)e^{-c\left( (s-s_i)\wedge (\el/2) \right)}
\end{equation}
In words, each position of the fronts on the times $s_i$ provides an interval at time $s$ where it is likely to find enough empty sites.
Moreover, we can control the front evolution during the intervals $[s_i,s_{i+1}]$ with Corollary \ref{vitessemax} and Lemma \ref{vitessemin} to find that:
\begin{align}\label{eq:proofzeros2}
&\P_\e \left( 0 \le X_{s_i} - X_{s_{i+1}} \le \bar{v}\Delta' \right) \ge 1 - e^{-c\Delta'}, \\
\label{eq:proofzeros3}
&\P_\e \left( 0 \le X_{s} - X_{s_{n}} \le \bar{v}\Delta \right) \ge 1 - e^{-c\Delta}.
\end{align}
Our choice of $\Delta$ and $\Delta'$ ensures that under these events, the intervals at time $s$ appearing in \eqref{eq:proofzeros1} overlap, that is to say:
$$\bigcup_{i=0}^n [X_{s_i} - \underline{v}(s-s_i), X_{s_i} + \underline{v}(s-s_i)] \supset [X_s + \el/2, X_0 + \underline{v}s]. $$
We now take into account the zeros of the initial configuration to make this interval go all the way to the boundary. Let $x_0 = X_0 < x_1 < \dots < x_m \le y $ be empty sites in $\e_0$ between $X_0$ and $y$. We choose them so that for all $i$, $x_{i+1} - x_{i} \le \el$ and $y - x_m \le \el$, and their cardinality $m$ is minimal. This way we can have $m \le \dfrac{2 |X_0|}{\el}$.\\
By Lemma \ref{lem:contact}, we have for all $i$,
\begin{equation}\label{eq:proofzeros4}
\P_\e \left(\e_s \notin \h(x_{i} - \underline{v}s, 0 \wedge (x_{i} + \underline{v}s), \el/2) \right) \le C s e^{-c (s \wedge (\el/2))}
\end{equation}
Now since $2\underline{v}s \ge \el,$ then $(x_{i+1} - \underline{v}s) - (x_i + \underline{v}s) \le 0$ which guarantees that the intervals $[x_{i} - \underline{v}s, x_{i} + \underline{v}s]$ overlap.\\

We can now conclude from equations \eqref{eq:proofzeros1}, \eqref{eq:proofzeros2}, \eqref{eq:proofzeros3} and \eqref{eq:proofzeros4}:
\begin{align*}
\P_{\e}\left(\tilde{\e}_s \notin \h(0,-X_s + y,\el)\right) &\le C(n+1)se^{-c (s \wedge (\el/2))} + e^{-c\Delta} + ne^{-c\Delta'} + C m s e^{-c (s \wedge (\el/2))}\\
	&\le C (1 + |X_0|) \dfrac{s^2}{\el} e^{-c (s \wedge \el)}.
\end{align*}
In the last line, we bounded $n \le C \dfrac{s}{\el}$ with $C$ depending only on $q$. Also note that we artificially added a factor $s > 1$ to the last three terms of the first line so we have a nicer expression in the end. While non optimal, this result will be precise enough for our needs.
\end{proof}

\section{Relaxation results}\label{sec:4}

We give here several results of relaxation starting out of equilibrium. They build upon the results of \cite{BCM13}. First, let us recall a proposition about the relaxation of the FA-1f process on a finite interval with zero boundary conditions.

\begin{prop}[{\cite[Corollary 3.3]{BDT19}}]\label{relaxationfinie}
Let $q > 1/2$, $L > 0$ and $f$ a bounded function with support contained in $[1,L]$ such that $\mu(f)=0$. If $L \le e^{t^\alpha}$ for some $\alpha < 1/2$, then there exists $c'=c'(\alpha,q) >0$ such that, if $\s_0 \in \h(0,L+1,\sqrt{t})$, 
$$|\E_{\s_0}[f(\s_t)]| \le \frac{1}{c'}||f||_\infty e^{-c'\sqrt{t}}$$
\end{prop}

In words, this proposition states that if the process starts in $\h(0,L+1,\el)$, then it is mixed after a time $\el^2$. Note that this holds for $q > 1/2$ which covers a larger regime than $q > \bar{q}$ because it does not rely on the coupling with the threshold contact process. Actually, we can combine this proposition with Lemma \ref{lem:contact} to get a better relaxation speed for the supercritical regime and have a mixing after a time $ \sim \frac{\el}{2\underline{v}}$ instead of $\el^2$:

\begin{prop}\label{relaxationfinie2}
Let $q > \bar{q}$, $0< \beta < 1/2$, $L > 0$, $\el >0$ and $f$ a bounded function with support contained in $[1,L]$ such that $\mu(f)=0$. If $L \le e^{\el^\alpha}$ for some $\alpha < \beta$, then there exist $c,C >0$ depending only on $q$ and $\alpha$ such that, if $\s_0 \in \h(0,L+1,\el)$  and $t = \frac{\el}{2\underline{v}} + \el^{\beta}$,
$$|\E_{\s_0}[f(\s_{t})]| \le C ||f||_\infty L e^{-c\el^{\beta/2}}.$$
\end{prop}

\begin{proof}
Set $t_1 = \frac{\el}{2\underline{v}}$.
Let $0 = x_1 < \dots < x_p = L+1$ a sequence of sites in $[0,L+1]$ such that:
\begin{itemize}
\item $\forall 1 \le i \le p, \ \s_0(x_i) = 0$,
\item $\forall 1 \le i \le p-1, \ x_{i+1} - x_{i} \le \el$,
\item $p \le \lceil \frac{2L}{\el} \rceil$.
\end{itemize}
Thanks to Lemma \ref{lem:contact}, we have
\begin{align*}
\P\left(\s_{t_1} \notin \h(0,L+1,\el^{\beta/2})\right) &\le \P\left( \bigcup_{i=1}^p \{ \s_{t_1} \notin \h((x_i - \underline{v}t_1)\vee 0, (x_i + \underline{v}t_1) \wedge (L+1) , \el^{\beta/2} / 2) \} \right)\\
 	&\le \sum_{i=1}^p \P\left(\s_{t_1} \notin \h(x_i - \underline{v}t_1, x_i + \underline{v}t_1, \el^{\beta/2} / 2) \right)\\
 	&\le Cpt_1 e^{-c (t_1 \wedge (\el^{\beta/2} / 2))}\\
 	&\le C'L e^{-c' \el^{\beta/2}}. 
\end{align*}
Now we can use Proposition \ref{relaxationfinie} (with $\s_{t_1} \in \h(0,L+1,\sqrt{\el^\beta})$, and $L < e^{\el^\beta}$ with $\beta < 1/2$) and Markov property:
\begin{align*}
\left| \E_{\s_0}[f(\s_{t})] \right| &\le \left| \E_{\s_0}[f(\s_t)\mathds{1}_{\s_{t_1} \in \h(0,L+1,\el^{\beta/2})}] \right| + ||f||_\infty \P\left(\s_{t_1} \notin \h(0,L,\el^{\beta/2})\right) \\ 
	&\le \E_{\s_0}[\mathds{1}_{\s_{t_1} \in \h(0,L+1,\el^{\beta/2})} |\E_{\s_{t_1}}[f(\s_{t-t_1})] |] + C'||f||_\infty L e^{-c' \el^{\beta/2}} \\
	&\le \frac{1}{c''} ||f||_\infty e^{-c'' \el^{\beta/2}} + C'||f||_\infty L e^{-c' \el^{\beta/2}}.
\end{align*}
\end{proof}

A notable result comes from the particular case $\el = L+1$ and $\beta = 1/4$. Indeed, any configuration $\s_0 \in \Omega_L$ belongs to $\h(0,L+1,L+1)$.

\begin{cor}
There exist $c,C >0$, such that for any $L>0$, if $t = \frac{L+1}{2\underline{v}} + (L+1)^{1/4}$,
$$\sup_{\s_0 \in LO_{[0,L]}} || \mu_t^{\s_0} - \mu ||_{TV} \le CLe^{-cL^{1/8}}.$$
\end{cor}

\begin{rem}
This results proves that the mixing time is always bounded by $\frac{L}{2\underline{v}} + o(L)$. This gives a first scale for $t_{\text{mix}}$ but the constant $\frac{1}{2 \underline{v}}$ will turn out to be too large in general, as we will see in the final section.
\end{rem}


\section{Central limit theorem for FA-1f on $\Z_-$}\label{sec:5}

We now aim to prove a central limit theorem for the front of the FA-1f process on $\Z_-$. Our result differs from the one proved in \cite[Theorem 2.2]{BDT19} in two ways. First, it studies an FA-1f process on the half-line as opposed to $\Z$. This change forces us to take into account the boundary condition in the origin, tough we will see that the front gets far from the origin so this barely makes a difference in the proof. The second change is a \emph{uniformity} result with respect to the initial configuration. This requires to study carefully the original proof and tweak a few key arguments. For a first read, one can take Theorem \ref{TCL} for granted and skip to the next section in order to get to the main result.

\begin{thm}\label{TCL}
Let $q>\bar{q}$. There exist $v > 0$ and $s \ge 0$ such that for all $\e \in LO_-$, and all real numbers $a < b$, 
\begin{equation}\label{Eq:TCLuniforme}\underset{\e \in LO_-}{\sup} \left| \P_\eta\left(a \le \dfrac{X_t - vt}{\sqrt{t}} \le b\right) - \P(a \le N \le b) \right|\underset{t \rightarrow \infty}{\longrightarrow} 0,
\end{equation}
with $N \sim \mathcal{N}(0,s^2)$. 
The constants $v$ and $s$ are the same as in \cite[Theorem 2.2]{BDT19}.
\end{thm}

\begin{rem}
In this theorem, the variance $s^2$ of the normal law can be zero. For this more favorable case, the law $\mathcal{N}(0,0)$ is simply $\delta_0$, so that we have $\underset{\e \in LO_-}{\sup} \left| \P_\eta\left(-a \le \dfrac{X_t - vt}{\sqrt{t}} \le a\right) - 1 \right|\underset{t \rightarrow \infty}{\longrightarrow} 0$ for all $a >0$.
\end{rem}

To prove this theorem, we first need to study the convergence of the law behind the front, which is the topic of the following theorem. 

\begin{thm}\label{ergo}
Let $q>\bar{q}$. The process seen from the front has a unique invariant measure $\nu$. This measure is the same as in \cite[Theorem 2.1]{BDT19}. There exist $d^*, c >0$ such that for all $\e_0 \in LO_-$, for $t$ large enough, 
$$ ||\tilde{\mu}^{\e_0}_t - \nu ||_{[0,d^*t]} \le \exp\left(-c e^{(\log t)^{1/4}}\right), $$
where $\tilde{\mu}^{\e_0}_t$ is the distribution of the configuration seen from the front at time $t$, starting from $\e_0$ and $||\cdot-\cdot||_\Lambda$ is the total variation distance on $\Lambda$.
\end{thm}

The proof of this result is almost identical to the original proof in \cite{BDT19}. It relies on a rather technical coupling argument that only requires tiny adjusments due to  the presence of a boundary in $\Z_-$. The adapted version of the proof can be found in Appendix \ref{sec:app}.

\subsection{A Central Limit Theorem}

From Theorem \ref{ergo}, we will now get a central limit theorem as in \cite{BDT19}. Here, we state a general central limit theorem that we will apply to the increments of th front in the next paragraph. It resembles \cite[Theorem A.1]{BDT19} but adds a uniformity result.

\begin{thm}\label{TCLgen}
Let $(\s_t)$ be a Markov process (in continuous time) on a probability space $(\Omega,\mathcal{F},\P)$ and $(\mathcal{F}_t)$ the adapted filtration. Let $(X_i)_{i\ge 1}$ be real random variables satisfying the following hypotheses:
\begin{enumerate}
\item \begin{enumerate}
	\item $\underset{\s \in \Omega}{\sup} \ \underset{n\in \N}{\sup} \ \E_\s[X_n^2] < \infty;$
	\item for every $i \ge 1$, $X_i$ is measurable w.r.t. $\mathcal{F}_i$;
	\item for every $k,n\ge 1$, $f : \R^n \rightarrow \R$ measurable such that $\underset{\s}{\sup} \ \E_\s[f(X_1,...,X_n)] < \infty$, for all initial $\s$, we have the Markov property
	\begin{equation}\E_\s[f(X_k,...,X_{k+n-1}) \ | \ \mathcal{F}_{k-1}] = \E_{\s_{k-1}}[f(X_1,...,X_{n-1})];
	\end{equation}
	\end{enumerate}
\item There exists a decreasing function $\Phi$, constants $C,c^* \ge 1$ and $v \in \R$ and a measure $\nu$ such that
		\begin{enumerate}
			\item $\underset{n\rightarrow\infty}{\lim} e^{(\log n)^2} \Phi(n) = 0$;
			\item for every $i \ge 1$, $\E_\nu[X_i] = v$;
			\item for every $k$, $f: \R \rightarrow \R$ s.t. $e^{-|x|}f^2(x) \in L^1(\R)$, we have $\underset{\s}{\sup} \ \E_\s[f^2(X_1)] < \infty$ and 
				\begin{equation}
				\underset{\s}{\sup} \ |\E_\s [f(X_k)] - \E_\nu[f(X_1)]| \le C(f) \Phi(k);	
				\end{equation}
			\item for every $k,n$ and $f$ such that $f: \R \rightarrow \R$ s.t. $e^{-|x|}f^2(x) \in L^1(\R)$,
				\begin{equation}
				\underset{\s}{\sup} \ |\Cov_\s [f(X_k),f(X_{n})] - \Cov_\nu[f(X_1), f(X_{n-k+1})]| \le C(f) \Phi(k);	
				\end{equation}
				\begin{equation}\label{eq:cov}
				\underset{\s}{\sup} \ |\Cov_\s [f(X_k),f(X_{n})]| \le C \ \underset{\s}{\sup} \ \E_\s [f(X_1)^2]^{1/2} \ \Phi(n-k);	
				\end{equation}
				
			\item for every $k,n$ such that $k \ge c^* n$ and any bounded function $F : \R^n \rightarrow \R$,
				\begin{equation}
				\underset{\s}{\sup} \ |\E_\s [F(X_k,...,X_{k+n-1})] - \E_\nu[F(X_1,...,X_{n})]| \le C ||F||_\infty \Phi(k).
				\end{equation}
		\end{enumerate}
\end{enumerate}

Then there exists $s \ge 0$ such that 
$$\dfrac{\sum_{i=1}^n X_i - vn}{\sqrt{n}} \overset{\mathcal{L}}{\longrightarrow} \mathcal{N}(0,s^2).$$ 

Moreover, this convergence holds uniformly in the initial configuration in the following sense:

$$\forall a<b, \ \underset{\s}{\sup} \ | \P \left( a \le \dfrac{\sum_{i=1}^n X_i - vn}{\sqrt{n}} \le b \right) - \P \left( a \le N \le b \right) | \underset{n \rightarrow \infty}{\longrightarrow} 0, $$

where $N \sim \mathcal{N}(0,s^2)$.

\end{thm}

Note that the hypothesis 2.(d) differs from the one in \cite{BDT19}. It is the correct hypothesis and should be corrected in the original article. It was however verified in the application therein, see \cite[Lemma 7.3]{BDT19}.

\begin{proof}
We first focus on bounded random variables. Let $(X_i)_{i \ge 1}$, satisfying the hypotheses. Let us define like in \cite{BDT19} $Y_i := X_i - \E_\nu [X_1]$, $\el_n = n^{1/3}$, $S_n = \sum_{i=1}^n Y_i$, $S_{j,n} = \sum_{i=1}^n \mathds{1}_{|k-j| \le l_n} Y_i$ and finally $\alpha_n = \sum_{i=1}^n \E_\s[Y_j S_{j,n}]$. and assume each $Y_i$ is bounded. We follow the same line of arguments, replacing the "Bolthausen Lemma" \cite[Lemma A.3]{BDT19} with an adapted version:

\begin{lem}\label{Bol}
Let $(\nu_n^\s)_{n\ge 0}$ be a family of probability measures on $\R$ such that:
\begin{enumerate}
\item $\underset{\s}{\sup} \ \underset{n}{\sup}\int |x|^2 d\nu_n^\s(x) < \infty$
\item $\forall R >0$, $\underset{n \rightarrow \infty }{\lim} \underset{\s}{\sup} \underset{|\lambda| \le R }{\sup} \left| \int (i\lambda -x)e^{i\lambda x} d\nu_n^\s(x) \right| = 0$
\end{enumerate}
Then for all $\s \in \Omega$, $(\nu_n^\s)_{n\ge 0}$ converges to the standard normal law. Furthermore, for all continuous bounded function $f$: 
\begin{equation}\label{eq:Bol1}
\underset{n \rightarrow \infty }{\lim} \ \underset{\s}{\sup} \left| \nu_n^\s(f) - \mu(f) \right| = 0, 
\end{equation}
where $\mu = \mathcal{N}(0,1)$.\\
Moreover, for all real numbers $a < b$,
\begin{equation}\label{eq:Bol2}
\underset{n \rightarrow \infty }{\lim} \ \underset{\s}{\sup} \left| \nu_n^\s(\mathds{1}_{[a,b]}) - \mu(\mathds{1}_{[a,b]})   \right| = 0
\end{equation}
\end{lem}

\begin{proof}[Proof of Lemma \ref{Bol}]
Suppose that there is no uniformity in \eqref{eq:Bol1}. Then there exist $\varepsilon >0$, an increasing sequence of integers $(k_n)$ and a sequence $(\s_n)$ such that:
$$\left| \nu_{k_n}^{\s_n}(f) - \mu(f) \right| > \varepsilon $$
Then the sequence of measures $(\nu_{k_n}^{\s_n})$ clearly satisfies the hypothesis of the original Lemma \cite[Lemma A.3]{BDT19} so should converge to a standard normal law, which is a contradiction with the above equation.
\end{proof}

Looking at the proof of Theorem A.1 in \cite{BDT19}, the distributions of $(\dfrac{S_n}{\sqrt{\alpha_n}})_{n \ge 0}$ satisfy the stronger hypotheses (in particular item 2) of the above Lemma. Consequently, we have the conclusion of Theorem \ref{TCLgen} for bounded variables.\\

We now give details on how to extend the bounded case to the general case. We no longer assume the $(Y_i)_{i \ge 0}$ to be bounded. Let $a \in \R$ be a real number and $\varepsilon >0$. Let $N \ge 0$ be an integer that will be fixed later, and define the truncation operator $T^N(x) := \max\{ \min(x,N),-N \}$, and the remainder $R^N(x):= x - T^N(x)$.\\

Then we can estimate:

\begin{align}
\P_\s &\left(\dfrac{\sum_{i=1}^n Y_i}{\sqrt{n}} \ge a \right) = \P_\s \left(\dfrac{\sum_{i=1}^n (T^N(Y_i) - \E_\nu[T^N(Y_i)]) + (R^N(Y_i) - \E_\nu[R^N(Y_i)])}{\sqrt{n}} \ge a \right) \nonumber\\
&\le \P_\s \left( \dfrac{\sum_{i=1}^n (T^N(Y_i) - \E_\nu[T^N(Y_i)])}{\sqrt{n}} \ge a - \varepsilon \right) + 2 \P_\s \left( \dfrac{|R^N(Y_i) - \E_\nu[R^N(Y_i)]|}{\sqrt{n}} \ge \varepsilon \right). \label{eq:reste}
\end{align}

We first handle the last term:

\begin{align*}
\P_\s \left( \dfrac{|R^N(Y_i) - \E_\nu[R^N(Y_i)]|}{\sqrt{n}} \ge \varepsilon \right) &\le \dfrac{1}{\varepsilon^2 n} \Var_\s \left(\sum_{i=1}^n R^N(Y_i)\right)\\
&\le \dfrac{1}{\varepsilon^2 n} \left( \sum_{i=1}^n \Var_\s(R^N(Y_i)) + \sum_{i\ne j} \Cov_\s(R^N(Y_i),R^N(Y_j))  \right).
\end{align*}

\begin{itemize}
\item For any $i$, note that we can easily bound $\E_\s [R^N(Y_i)^2]$:
\begin{align*}\E_\s(R^N(Y_i)^2) &= \E_\s[(Y_i-N)^2\mathds{1}_{Y_1\ge N}] + \E_\s[(Y_i+N)^2\mathds{1}_{Y_1\le -N}]\\
	&\le 4\sqrt{2} (\E_\s[Y_i^4] + N^4)^{1/2} \P_\s(|Y_i| \ge N)^{1/2}\\
	&\le 4\sqrt{2} (C + N^4)^{1/2} \P_\s(|Y_i| \ge N)^{1/2}\\
	&\le C' (C + N^4)^{1/2} e^{-N/4},
\end{align*}
where $C$ does not depend on $\s$ or $i$. The variables $Y_i$ have a bounded forth moment thanks to hypothesis 2.(c). The last bound comes from the exponential Chebychev inequality and hypothesis 2.(c).\\
From this we can conclude that 
$$ \dfrac{1}{\varepsilon^2 n} \sum_{i=1}^n \Var_\s(R^N(Y_i)) \le C \varepsilon^{-2}e^{-cN},$$
with $c,C >0$ independent of $n,N,\s$.

\item From assumption 2.(d) (equation \eqref{eq:cov}), we get that:
\begin{align*}
\sum_{i\ne j}^n \Cov_\s(R^N(Y_i),R^N(Y_j)) &\le \sum_{i=1}^n \sum_{j=1}^\infty \underset{\s}{\sup} \ \E_\s [R^N(Y_1)^2]\Phi(|j-i|) \\
	&\le C e^{-cN} \sum_{i=1}^n \sum_{j=1}^\infty \Phi(|j-i|) \\
	&\le C' n e^{-cN},
\end{align*} 
where we used the bound on $\E_\s[R^N(Y_1)^2]$ previously found for the second inequality, and the fact that $\sum_j \Phi(j) < \infty$ for the last one.\\
\end{itemize}

Putting the last two inequalities into \eqref{eq:reste}, we get:

$$\P_\s \left(\dfrac{\sum_{i=1}^n Y_i}{\sqrt{n}} \ge a \right) \le \P_\s \left( \dfrac{\sum_{i=1}^n (T^N(Y_i) - \E_\nu[T^N(Y_i)])}{\sqrt{n}} \ge a - \varepsilon \right) + C \varepsilon^{-2} e^{-cN}.$$

Let $N$ be such that $C \varepsilon^{-2} e^{-cN} \le \varepsilon$. We now use our theorem on the bounded variables $\left( T^N(Y_i) - \E_\nu[T_N(Y_i)] \right)$ to find:

$$\underset{n \rightarrow \infty}{\limsup} \ \underset{\s}{\sup} \left(\P_\s \left(\dfrac{\sum_{i=1}^n Y_i}{\sqrt{n}} \ge a \right) - \P(N \ge a- \varepsilon)\right) \le \varepsilon, $$

which in turn shows that 

$$\underset{n \rightarrow \infty}{\limsup} \ \underset{\s}{\sup} \left(\P_\s \left(\dfrac{\sum_{i=1}^n Y_i}{\sqrt{n}} \ge a \right) - \P(N \ge a)\right) \le 0. $$

It now remains to prove the reverse inequality. $\P_\s \left(\dfrac{\sum_{i=1}^n Y_i}{\sqrt{n}} \ge a \right)$ can be bounded below by:

\begin{align*}
 &\P_\s \left(\dfrac{\sum_{i=1}^n T^N(Y_i)-\E_\nu[T^N(Y_i)]}{\sqrt{n}} \ge a + \varepsilon \ \textup{and} \ \left|\dfrac{\sum_{i=1}^n R^N(Y_i)-\E_\nu[R^N(Y_i)]}{\sqrt{n}}\right| \le \varepsilon  \right)\\
	&\ge \P_\s \left(\dfrac{\sum_{i=1}^n T^N(Y_i)-\E_\nu[T^N(Y_i)]}{\sqrt{n}} \ge a + \varepsilon\right) - \P_\s \left(\left|\dfrac{\sum_{i=1}^n R^N(Y_i)-\E_\nu[R^N(Y_i)]}{\sqrt{n}}\right| \ge \varepsilon  \right)
\end{align*}

From here, we bound the last term as we have done previously, and find this time:

$$\underset{n \rightarrow \infty}{\limsup} \ \underset{\s}{\sup} \left(\P_\s \left(\dfrac{\sum_{i=1}^n Y_i}{\sqrt{n}} \ge a \right) - \P(N \ge a)\right) \ge 0, $$
which concludes the proof.
\end{proof}

\subsection{Proof of Theorem \ref{TCL}}

The aim of the following four lemmas is to justify the various hypotheses of Theorem \ref{TCLgen} on the increments $\xi_n := X_n - X_{n-1}$, where $X_n$ denotes the front at time $n$ of an FA-$1$f process on $\Z_-$ started at an arbitrary configuration $\e \in LO_-$. They can be proved in the exact same way as in \cite{BDT19}, with a minor change explained in the proof of Lemma \ref{LemTCL}. In the following, $\nu$ denotes the measure defined in Theorem \ref{ergo}.

\begin{lem}
For $f: \Z \rightarrow \R$ such that $e^{-|x|}f(x)^2 \in L^1,$ we have
\begin{equation}
\underset{\e \in LO_-}{\sup} \E_\e[f(\xi_1)^2] = c(f) < \infty.
\end{equation}
\end{lem}


\begin{lem}\label{LemTCL}
There exists $\gamma > 0$ such that for $f : \Z \rightarrow R$ with $e^{-|x|}f(x)^2 \in L^1(\R)$,
\begin{equation}
\underset{\e \in LO_-}{\sup} | \E_\e [f(\xi_n)] - \E_\nu[f(\xi_1)] | \le C(f)e^{-\gamma e^{(\log n)^{1/4}}}.
\end{equation}
\end{lem}

\begin{proof}
Here, we define as in \cite{BDT19}, the configuration $\Phi_t(\e)$ to be equal to $\e$ on $[X(\e),X(\e) + d^*t]$ and $1$ elsewhere (except at the origin). Let $X$, (resp.\ $\tilde{X}$) the front of the configuration starting form $\e$ (resp.\ $\Phi_t(\e)$), both being coupling via the standard coupling. In our case, it is possible that the interval $[X(\e),X(\e) + d^*t]$ contains the origin, which can mess up the original argument.
To prevent this, let us note the following fact:\\
if $t \ge 0, \e \in LO_-,$ the event $X \ne \tilde{X}$ is a subset of $\mathcal{A}_t := F(0,d^*t,1) \cup \{X_t \ge -\underline{v}t \}$.\\
For $\e \in LO_-$, $f$ an appropriate function, and $n \ge 0$, we get:
\begin{align*}
\E_\e [f(\xi_1)] - \E_{\Phi_{n-1}(\e)}[f(\xi_1)] &= \E[(f(X) - f(\tilde{X})) \mathds{1}_{X \ne \tilde{X}}] \\
	&\le \sqrt{\E \left[ (f(X) - f(\tilde{X}))^2 \right]} \sqrt{\P(\mathcal{A}_{n-1})}.
\end{align*}
 
Now note that $\P(\mathcal{A}_{n-1}) \le e^{-\frac{d^*(n-1)}{2}} + Ae^{-B(n-1)} = O(e^{-cn})$.\\
From there, we can conclude like in \cite{BDT19}.
\end{proof}

\begin{lem}
There exists $\gamma >0$ such that, for $f : \N \rightarrow \R, e^{-|x|}f(x)^2 \in L^1(\R)$ and $j<n$ two positive integers, 
\begin{enumerate}
\item $\underset{\e \in LO_-}{\sup} \left| \Cov_\e[f(\xi_j),f(\xi_n)] \right| \le C(f)e^{-\gamma e^{(\log(n-j))^{1/4}}},$ and the same holds for the covariance under $\nu$;
\item for $j \ge \frac{\bar{v}}{d^*}(n-j)$,
$$\underset{\e \in LO_-}{\sup} \left| \Cov_\e[f(\xi_j),f(\xi_n)] - \Cov_\nu[f(\xi_1),f(\xi_{n-j+1})] \right| \le C(f)e^{-\gamma e^{(\log j)^{1/4}}}. $$
\end{enumerate}
\end{lem}

\begin{lem}
For any $k,n \in \N$ such that $d^*(k-1) \ge \bar{v}n$ and any bounded function $F: \R^n \rightarrow \R$
$$\underset{\e \in LO_-}{\sup} \left| \E_\e [F(\xi_k,...,\xi_{k+n-1})] - \E_\nu[F(\xi_1,...,\xi_n)] \right| = O\left( ||F||_\infty e^{-\gamma e^{(\log k)^{1/4}}} \right). $$
\end{lem}

From here, we can apply Theorem \ref{TCLgen} and conclude exactly like in \cite{BDT19}.

\section{Cut-off for FA-1f}\label{sec:6}

We are now ready to prove the main result. Let $\Lambda = [1,L]$, we fix $q > \bar{q}$ the parameter of the FA-1f process. Every constant introduced in this section implicitly depends on $q$. 

We split the proof into two cases. First, we tackle initial configurations with macroscopic sub-intervals of occupied sites. In this section, we shall call such an interval a \emph{particle cluster}. To tackle these, we need to study closely the fronts of each particle cluster that are going inward. We will see that after the time it takes for them to meet, the configuration created enough empty sites to relax to equilibrium quickly. Next, we study the initial configurations with no such particle cluster, that is configurations in $\h(0,L+1,\el)$ for a certain threshold $\el(L)$. We handle the latter category with softer arguments using the coupling of FA-1f and the contact process. 

\subsection{Configurations with macroscopic particle clusters}
Define, for any $L>0$ and $0 < \delta < 1$, $$\Omega^\delta = \{ \s \in \Omega_\Lambda \ | \ \s \in \h(0,L+1,\delta L) \}.$$
For $\s \in \Omega_{\Lambda}$, recall that $B(\s) := \max \{ h \ge 0 \ | \ \exists x \in [0,L-h], \ \restriction{\s}{[x+1,x+h]} \equiv 1 \}$ is the size of the largest particle cluster in $\s$.
The following result holds.

\begin{prop}\label{prop:cutofffull}
Let $\delta >0$, $\s_0 \in (\Omega^\delta)^c$ and $\varepsilon >0$. There exists $a = a(\varepsilon) >0$ such that,\\
if $t = \dfrac{B(\s_0)}{2v} + \dfrac{7av}{\underline{v}\delta}\sqrt{L}$, then:
$$ ||\mu_{t}^{\s_0}-\mu||_{TV} \le \dfrac{v^2}{\underline{v}^2\delta^2}\varepsilon + \phi(L),$$
with $\phi(L)$ independent of $\s_0$ and $\underset{L \rightarrow \infty}{\lim} \phi(L) = 0$.
\end{prop}

The proof of this proposition is divided into four lemmas. To make things easier, let us first introduce the sketch of the proof in words. In the case we are studying now, the initial configuration has at least one macroscopic particle cluster of size $\ge \delta L$. If we watch the evolution of the process in such an interval during a short period of time, we see that the two endpoints of this particle cluster behave just like fronts of FA-1f processes in half-lines (one left-oriented, the other one right-oriented). As a consequence, thanks to the Central Limit Theorem \ref{TCL} we can find an explicit time after which the fronts are much closer to each other, at a certain threshold distance $d$ of order $\sqrt{L}$. At this point, the space between the original positions of the fronts and the current one should contain a lot of zeros thanks to the Zeros Lemma \ref{zeros1}. We repeat this argument for every initial particle cluster, and thanks to appropriate choices of constants and coupling arguments  handling the space in between the clusters, we end up with a configuration in $\h(\Lambda, d)$ with high probability. From there, it will only remain to apply the relaxation result from Section \ref{sec:4} to conclude.

Before diving into the proof, we need to set a general framework. The following lemmas aim to explain in detail how a configuration with some particle clusters can turn the smallest of them into an interval with a lot of zeros.\\
Let us fix an arbitrary set of disjoint intervals $\Lambda_1,\dots,\Lambda_r \subset \Lambda$. We define their endpoints: $\Lambda_k = [a_k,b_k]$. Fix $\varepsilon >0$, and let $d=6a\sqrt{L}$, with $a$ a constant depending on $\varepsilon$ that will be determined later.\\
We now define a class of configurations that we consider in our proof. We say that a configuration $\s \in \Omega_\Lambda$ is in the class $\mathcal{C}$ if for all $1 \le k \le r$, there exist $X_k < Y_k \in \Lambda_k$ such that
\begin{enumerate}
\item $\s(X^k) = \s(Y^k) = 0,$
\item $\restriction{\s}{(X^k,Y^k)} \equiv 1,$
\item $\s \in \h(a_k,X^k,d) \cap \h(Y^k,b_k,d).$
\end{enumerate}
Note that $X^k$ and $Y^k$ are uniquely determined if $\s \in \mathcal{C} \setminus \h(\Lambda_k,d)$. If $\s \in \h(\Lambda_k,d)$, there can be multiple choices of $X^k$ and $Y^k$, but this case will be excluded in the following. 
Let us introduce a few functions on $\mathcal{C}$ (see Figure \ref{fig:lemma0}). Fix $\s \in \mathcal{C}$.
\begin{itemize}
\item $\kappa(\s) = \{ k \ | \ \s \notin \h(\Lambda_k,d) \}$ is the set of indices such that $Y^k-X^k > d$. We write $\kappa$ if the configuration $\s$ is clear from the context. 
\item $p(\s) = \# \kappa$ is the cardinality of $\kappa$,
\item $\el(\s) = \underset{k \in \kappa}{\min}(Y^k-X^k)$,
\item $t(\s) = \frac{\el(\s)}{2v} - \frac{2a}{v} \sqrt{\el(\s)}$. We set $t(\s)=0$ if $p(\s)=0$. 
\end{itemize}
Having set that, we define the random variable $F(\s) = \s_{t(\s)}$. That is, $F(\s)$ is the result of the FA-1f process with initial configuration $\s \in \mathcal{C}$ after a time $t(\s)$. The aim of the random function $F$ is to turn the smallest particle cluster $[X^k,Y^k]$  of size $\ge d$ into an interval in $\h(\cdot,d)$. Note that this construction depends on the family of intervals $\Lambda_1,...,\Lambda_r$ that is for now arbitrary.\\

\begin{figure}[h]
{\centering\input{dessin_lemma0}\par}
\caption{A configuration in the set $\mathcal{C}$. Dashed zones represent zones in $\h(.,d)$. Here, we have $\kappa(\s) = \{1,3\}$, $p(\s)=2$ and $\el(\s) = Y^1 - X^1$.} 
\label{fig:lemma0}
\end{figure}
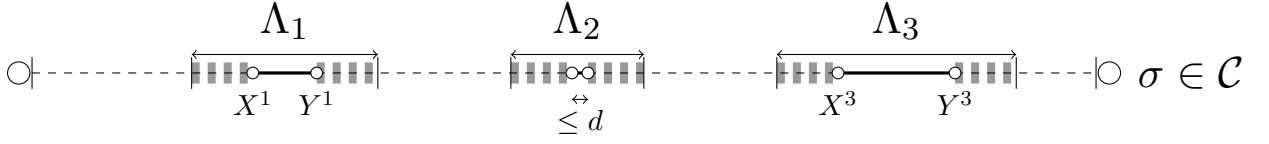


Let $\s \in \mathcal{C}$ and $(\s_t)_{t \ge 0}$ be a FA-1f process starting from $\s$. Define $X^k_t$ (resp.\ $Y^k_t$) as the position of the rightmost (resp.\ leftmost) zero in $[0, M^k]$ (resp.\ $[M^k, L+1]$) in $\s_t$, with $M^k := \dfrac{1}{2}(X^k+Y^k)$. We have $X^k_0 = X^k$ and $Y^k_0 = Y^k$. As long as $X_t^k$ and $Y_t^k$ do not meet, we think of $X^k_t$ as the "front" of a configuration on the (right oriented) half-line.
Let $\mathcal{R}_t = \{\forall k \in \kappa, \ \forall s < t, \ X^k_s < M^k - 1 \ \text{and} \ Y^k_s > M^k +1 \}$ the event that none of the "fronts" reaches the midpoint of the interval. In particular, under $\mathcal{R}_t$, the front have not met. The following Lemma localizes these fronts at time $t(\s)$ using Theorem \ref{TCL}.

\begin{lem} \label{lem:TCL}
There exists $a >0$ depending only on $\varepsilon$ and $q$ such that, with the notations previously introduced, and $d = 6a \sqrt{L}$, the following holds. For all $\s \in \mathcal{C}$ such that $p(\s) \ge 1$,
\begin{multline}
\P\left( \forall k \in \kappa, \ X^k_{t(\s)} \in [X^k + v t(\s) - a \sqrt{\el(\s)}, X^k + vt(\s) + a \sqrt{\el(\s)}], \right. \\
\left. \ Y^k_{t(\s)} \in [Y^k + v t(\s) - a \sqrt{\el(\s)}, Y^k + v t(\s) + a \sqrt{\el(\s)}], \mathcal{R}_{t(\s)}\right) \ge 1 - r\varepsilon + r\phi(L)
\end{multline}
with $\phi(L)$ independent of $\s$ such that $\underset{L \rightarrow \infty}{\lim}\phi(L) = 0$.\\
In particular, for $k_0$ such that $Y^{k_0}-X^{k_0} = \el(\sigma)$, the event above implies that: 
\begin{equation*}\label{eq:disp.petiteboite}
0 \le Y^{k_0}_{t(\s)} - X^{k_0}_{t(\s)} \le d
\end{equation*}

From now on, we call $$I_k = [X^k + v t(\s) - a \sqrt{\el(\s)}, X^k + vt(\s) + a \sqrt{\el(\s)}]$$ and $$J_k = [Y^k + v t(\s) - a \sqrt{\el(\s)}, Y^k + v t(\s) + a \sqrt{\el(\s)}].$$

\end{lem}

\begin{figure}[h]
{\centering\input{dessin_lemma1}\par}
\caption{The expected behaviour of the smallest cluster of $\s$ during time $t(\s)$.} 
\label{fig:lemma1}
\end{figure}
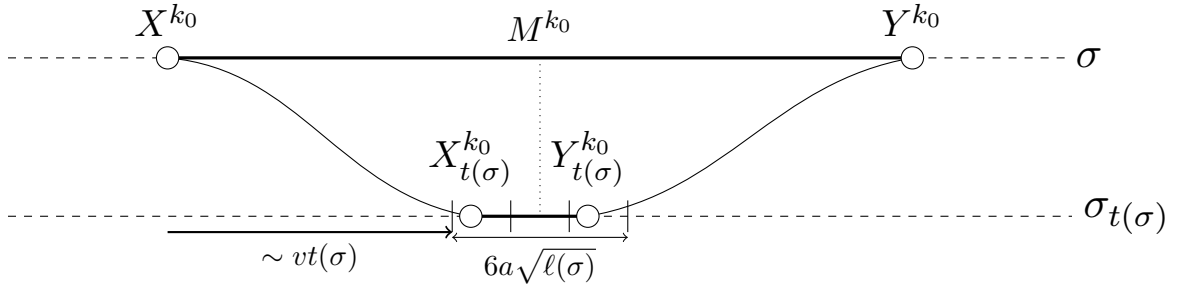

\begin{proof}
Define $\mathcal{R}^k_t = \{ \forall s \le t, \ X^k_s < M^k - 1 \ \text{and} \ Y^k_s > M^k +1 \}$ such that $\mathcal{R}_t = \bigcap_{1 \le k \le p(\s)} \mathcal{R}^k_t$. 
The key remark for what follows is that if the event $\mathcal{R}^k_t$ occurs, the fronts $X^k_s, Y^k_s$ behave like independent fronts of FA-1f processes on $\Z_-$ for all $s \le t$. To justify that, let us define auxiliary processes $(\hat{\s}^k_t)_{t \ge 0}$ and $(\check{\s}^k_t)_{t \ge 0}$ on $\Z_-$ as follows.
\begin{itemize}
\item $\forall x \le 0, \ \hat{\s}^k_0 (x) = \left\{ \begin{matrix} \s(-x) & \text{if} \ -x \le X^k, \\ 1 & \text{if} \ -x > X^k. \end{matrix} \right.$
\item $(\hat{\s}^k_t)_{t \ge 0}$ is the FA-1f process on $\Z_-$ starting from $\hat{\s}^k_0$ constructed with the standard coupling with respect to $(\s_t(\text{-}.))_{t \ge 0}$. We denote by $\hat{X}^k_t$ its front at time $t$.
\item $\forall x \le 0, \ \check{\s}^k_0 (x) = \left\{ \begin{matrix} \s(x+L+1) & \text{if} \ x+L+1 \ge Y^k, \\ 1 & \text{if} \ x + L +1  < Y^k. \end{matrix} \right.$
\item $(\check{\s}^k_t)_{t \ge 0}$ is the FA-1f process on $\Z_-$ starting from $\check{\s}^k_0$ constructed with the standard coupling with respect to $(\theta_{L+1}\s_t)_{t \ge 0}$. We denote by $\check{Y}^k_t$ its front at time $t$.
\end{itemize}

Let $\hat{\mathcal{R}^k_t} = \{ \forall s \le t, \hat{X}^k_s > -M^k + 1\}$, and $\check{\mathcal{R}^k_t} =  \{ \forall s \le t, \check{Y}
_s >  M^k - L \}$. Then, for all intervals $B, B' \subset \Lambda$, we have:
\begin{equation} \label{eq:frontssym}
\P_\s \left(X^k_t \in B, Y^k_t \in B' \ \textup{and} \ \mathcal{R}^k_t \right) = \P_{\hat{\s}^k_0} \left(\hat{X}^k_t \in -B \ \textup{and} \ \hat{\mathcal{R}}^k_t \right) \P_{\check{\s}^k_0} \left(\check{Y}^k_t \in \theta_{-L-1}(B') \ \textup{and} \ \check{\mathcal{R}}^k_t \right).
\end{equation}

We apply this with the intervals $I_k$ and $J_k$. The two factors in the right hand side are similar so we only study $\P_{\hat{\s}^k_0} \left(\hat{X}^k_t \in -I_k \ \textup{and} \ \hat{\mathcal{R}}^k_t \right)$. Note that $$\P_{\hat{\s}^k_0} \left(\hat{X}^k_t \in -I_k \ \textup{and} \ \hat{\mathcal{R}}^k_t \right) = \P_{\hat{\s}^k_0} \left(\hat{X}^k_t \in -I_k \right) - \P_{\hat{\s}^k_0} \left(\hat{X}^k_t \in -I_k \ \textup{and} \ (\hat{\mathcal{R}}^k_t)^c \right)$$

Thanks to Theorem \ref{TCL}, we have, for $N \sim \mathcal{N}(0,s^2)$,
\begin{align}
\P_{\hat{\s}^k_0} \left(\hat{X}^k_t \in -I_k \right) &=  \P_{\hat{\s}^k_0}\left( \dfrac{\hat{X}^k_t - \hat{X}_0^k - vt(\s)}{\sqrt{\el(\s)}/ (2v)} \in [-2va,2va] \right) \nonumber \\
		&\underset{L \rightarrow \infty}{\longrightarrow}  \P\left( N \in [-2va,2va] \right)
\end{align}
since $\el(\s) \ge d \rightarrow \infty$ as $L \rightarrow \infty$. This convergence is furthermore uniform in $\s$. 
We choose $a$ such that $\P\left( N \in [-2va,2va] \right) \ge \sqrt{1 - \varepsilon}$. Note that this choice only depends on $\varepsilon$ and $q$ (because the variance $s^2$ depends on $q$). This proves that
\begin{equation}\label{eq:TCL}
\P_{\hat{\s}^k_0} \left(\hat{X}^k_t \in -I_k \right) = \sqrt{1 - \varepsilon} + \phi_1(L)
\end{equation}
with $\phi_1(L)$ independent of $\s$ and $\underset{L \rightarrow \infty}{\lim}\phi_1(L) = 0$. We handle the term in $\check{Y}^k_t$ the same way and find the same kind of equation.
Note here that if $s^2 =0$, our estimate is even better as we would get $\P_{\hat{\s}^k_0} \left(\hat{X}^k_t \in -I_k \right) = 1- \phi_1(L)$.\\

Next, we show that if the fronts end up in the expected intervals at time $t(\s)$, then $\hat{\mathcal{R}}^k_{t(\s)}$ is very likely. We denote $\hat{M}^k = - M^k$ and we divide the time interval $[0,t(\s)]$ into intermediate times $0 < s_1 < ... < s_n = t(\s)$ with $n = \left\lceil \dfrac{2\bar{v}t(\s)}{a\sqrt{\el(\s)} - 2} \right\rceil$ such that $\Delta := s_{i+1} - s_i \le \dfrac{a}{2\bar{v}}\sqrt{\el(\s)} - \dfrac{1}{\bar{v}}$. Notice that:
\begin{multline*}
 \{\exists s < t(\s), \hat{X}^k_s \le \hat{M}^k + 1 \} \subset  \\
 \{\exists i, \hat{X}^k_{s_i} \le \hat{M}^k + \frac{a}{2}\sqrt{\el(\s)}  \} \cup \{\exists s < t(\s), \hat{X}^k_s \le \hat{M}^k + 1 \,\, \textup{and} \,\, \forall i, \hat{X}^k_{s_i} > \hat{M}^k + \frac{a}{2}\sqrt{\el(\s)}  \}.
\end{multline*}

First, with Corollary \ref{vitessemax} and with our choice of $\Delta$, we have that the second event is unlikely because of finite speed of propagation:
\begin{align*}
\P_\s \left(\exists s < t(\s), \hat{X}^k_s \le \hat{M}^k + 1 \right. & \left. \,\, \textup{and} \,\, \forall i, \hat{X}^k_{s_i} > \hat{M}^k + \frac{a}{2}\sqrt{\el(\s)} \right)\\
&\le \sum_{i=0}^{n-1} \P_\s \left( \exists  s \in [s_i,s_{i+1}], |\hat{X}^k_{s}-\hat{X}^k_{s_{i+1}}| \ge \frac{a}{2}\sqrt{\el(\s)} - 1 \right)\\
&\le \sum_{i=0}^{n-1} \P_\s \left( \exists s \in [s_i,s_{i+1}], |\hat{X}^k_{s}-\hat{X}^k_{s_{i+1}}| \ge \bar{v} (s_{i+1}-s_i)\right)\\
&\le Ce^{-c \sqrt{\el(\s)}} \le Ce^{-c \sqrt{d}} \hspace{10pt} \small{\textup{for some $C,c > 0$}}.
\end{align*}

Then, by the Markov property at time $s_i$, the probability that $\hat{X}^k_{s_i} \le \hat{M}^k + \frac{a}{2}\sqrt{\el(\s)}$ while $\hat{X}^k_{t(\s)}$ is to the right of that is bounded by the probability that the front is going backward for a time at least $\Delta$. This gives, thanks to Lemma \ref{vitessemin}:
\begin{multline*}
\P_\s \left(\hat{X}^k_{t(\s)} \in -I_k \,\, \textup{and} \,\, \exists i, \hat{X}^k_{s_i} \le \hat{M}^k + \frac{a}{2}\sqrt{\el(\s)} \right)\\
\le \sum_{i=0}^{n-1} \P_\s (\hat{X}^k_{t(\s)} > \hat{X}^k_{s_i}) \le nA e^{-B\Delta} \le A e^{-B'\sqrt{d}}.
\end{multline*}

Combining the previous inequalities shows that:
\begin{equation}\label{eq:eventRt}
\P_\s \left( \hat{X}^k_{t(\s)} \in -I_k \,\, \textup{and} \ (\hat{\mathcal{R}}^k_t)^c  \right) \le Ce^{-c\sqrt{d}},
\end{equation}
with $C,c > 0$ independent of $\s, L, \varepsilon$.

Combining Equations \eqref{eq:frontssym}, \eqref{eq:TCL} and \eqref{eq:eventRt}, since $d = 6a\sqrt{L}$, we get 
$$\P_\s \left(X^k_t \in I_k, \ Y^k_t \in J_k \ \textup{and} \ \mathcal{R}^k_t \right) \ge 1 - \varepsilon + \phi(L),$$
with $\phi(L)$ going to $0$ as $L \rightarrow \infty$, uniformly in $\s$.
By a union bound over all boxes $\Lambda_k$ for $k \in \kappa$, we get the expected result since the number of boxes is bounded by $r$.
\end{proof}

We have now localised the fronts in $F(\s)$, so we can use Lemma \ref{zeros1} to show that behind those fronts, with high probability, there are a lot of zeros (see Figure \ref{fig:lemma2}).

\begin{lem}\label{lem:zerosbehind}
Let $\s \in \mathcal{C}$ such that $p(\s) \ge 1$. With the previous notations, for all $k \in \kappa$, 
\begin{multline*}
\P_\s \left(X^{k}_{t(\s)} \in I_k, \ Y^{k}_{t(\s)} \in J_k, \ \mathcal{R}_{t(\s)} \right.\\
\left. \textup{and} \ \s_{t(\s)} \notin \left( \h(X^k,X^k_{t(\s)},d)\cap \h(Y^k_{t(\s)},Y^k,d) \right) \right) \le C \exp(-c\sqrt{L})
\end{multline*} 
for some $C,c > 0$ independent of every other parameter.
\end{lem}

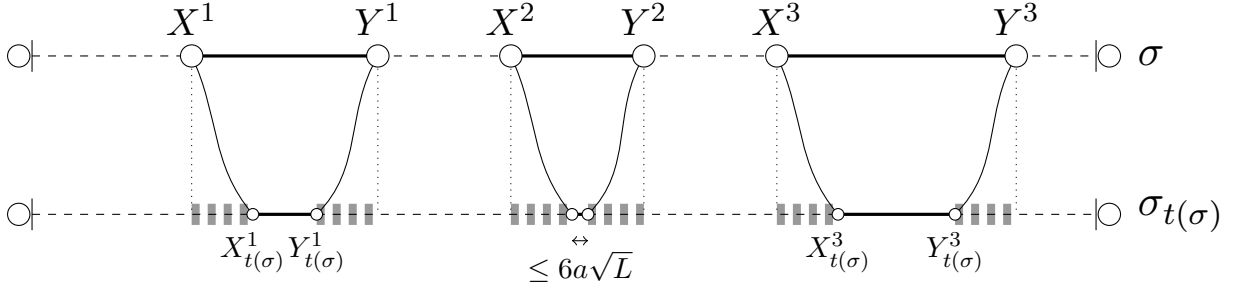
\begin{figure}[h]
{\centering\input{dessin_lemma2}\par}
\caption{The expected behaviour of the fronts during time $t(\s)$. Dashed zones are in $\h(.,d)$. We see that at time $t(\s)$, the configuration lost at least one particle cluster.}
\label{fig:lemma2}
\end{figure}

\begin{proof}
Given the event $\mathcal{R}_{t(\s)}$, we can use the previous coupling argument to show that:
\begin{multline}
\P_\s \left( X^{k}_{t(\s)} \in I_k, \ Y^{k}_{t(\s)} \in J_k, \ \mathcal{R}_{t(\s)} \ \textup{and} \ \s_{t(\s)} \notin \left( \h(X^k,X^k_{t(\s)},d)\cap \h(Y^k_{t(\s)},Y^k,d) \right) \right) = \\ 
\P_{\hat{\s}^k_0} \left( \hat{X}^k_t \in -I_k, \ \hat{\s}^k_{t(\s)} \notin \h(\hat{X}_t^k, \hat{X}^k_0, d), \hat{\mathcal{R}}_{t(\s)} \right) \\
\P_{\check{\s}^k_0} \left(  \check{Y}^k_t \in \theta_{-L}(J_k), \ \check{\s}^k_{t(\s)} \notin \h(\check{Y}^k_t, \check{Y}_0^k, d), \check{\mathcal{R}} _{t(\s)} \right).
\end{multline}

By Lemma \ref{zeros2} (with $y = \hat{X}_0^k$), we have that:
$$\P_{\hat{\s}^k_0} \left( \hat{\s}^k_{t(\s)} \notin \h(\hat{X}_t^k, \hat{X}^k_0,d) \right) \le C L^{5/2}\exp(-c \sqrt{L}).$$
We can get rid of the $L^{5/2}$ factor by taking $c$ smaller. The same inequality holds for $\check{\s}^k$ and therefore we get the claimed result.
\end{proof}

Now, we show that if an interval already contains enough zeros, it most likely will also at time $t(\s)$. This will ensure that once a cluster has disappeared and has left a lot of empty sites, the region will keep as many empty sites during any other iteration of $F$. It is actually a very general result that only relies on Lemma \ref{lem:contact}.

\begin{lem}\label{lem:goodstaysgood}
Let $\s \in \Omega_\Lambda$, $x,y \in \Lambda$, $d >0$ such that $\s \in \h(x,y,d)$. Then:
\begin{equation}
\P_\s( F(\s) \notin \h(x,y,d)) \le CL \exp(-c\sqrt{L}),
\end{equation}  
with $c,C$ independent of $d,x,y,\s$.
\end{lem}

\begin{proof}
We simply use the comparison with a supercritical contact process (Lemma \ref{lem:contact}) for every initial zero in $\restriction{\s}{(x,y)}$ with $\el = d/2$.
\end{proof}

So far, we have proved that with high probability after a time $t(\s)$, every front moves by a distance $\sim \dfrac{\el(\s)}{2v}$ (Lemma \ref{lem:TCL}), while creating zeros behind it (Lemma \ref{lem:zerosbehind}). This results in one box $k_0$ initially "blocked" ($Y^{k_0}-X^{k_0} > d$) to become filled with zeros. Meanwhile, every box $\Lambda_k$ that satisfies $\s \in \h(\Lambda_k,d)$ stays in $\h(\Lambda_k,d)$ after a time $t(\s)$ (Lemma \ref{lem:goodstaysgood}). As illustrated in Figure \ref{fig:lemma3}, all of this combines into a nice result about $F(\s)$.

\begin{lem}\label{lem:pdecroissant}
For any $\s \in \mathcal{C}$ such that $p(\s) \ge 1$,
$$\P\left(F(\s) \in \mathcal{C} \ \text{and} \ p(F(\s)) < p(\s)\right) \ge 1 - r\varepsilon - r \phi_2(L)$$
with $\phi_2(L)$ independent of $\s$ (but depending on $r$) and $\underset{L \rightarrow \infty}{\lim} \phi_2(L) = 0$.
\end{lem}

\begin{figure}[h]
{\centering\input{dessin_lemma3}\par}
\caption{The expected behaviour of the whole configuration during time $t(\s)$. Here, the interval $\Lambda_2$ kept its zeros and the particle cluster in $\Lambda_1$ shrunk to a size $\le d$.}
\label{fig:lemma3}
\end{figure}
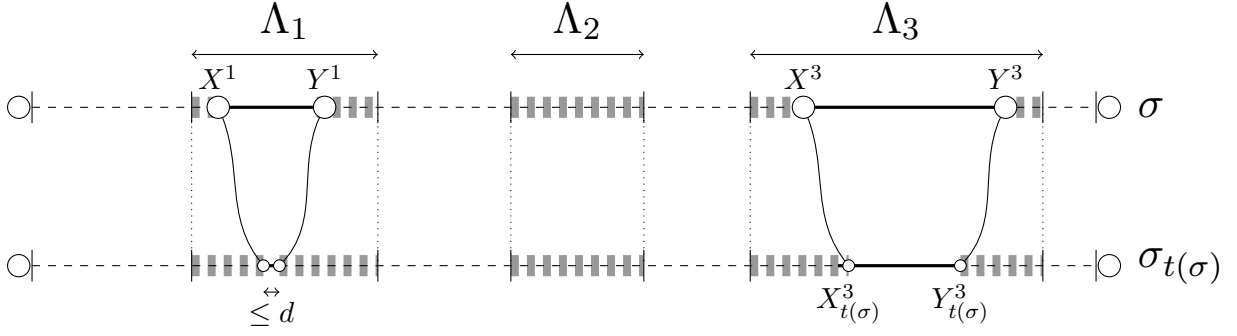

Before getting to the proof of Proposition \ref{prop:cutofffull}, let us make a final remark. On the event appearing in Lemma \ref{lem:TCL}, we can bound the size of every particle cluster in $F(\s)$:
\begin{equation}\label{eq:shrinkingclusters}
\forall k\in \kappa, \ Y^k_{t(\s)} - X^k_{t(\s)} \le Y^k - X^k - \el(\s) + 2a\sqrt{\el}
\end{equation}
Note that it is possible that in one iteration of $F$, two or more particle clusters go from a size $>d$ to a size $\le d$.

Let us now prove Proposition \ref{prop:cutofffull}.
\begin{proof}[Proof of Proposition \ref{prop:cutofffull}]
Fix $\delta >0$, $\s_0 \in (\Omega^\delta)^c$ and $\varepsilon >0$. We also define $a$ as in Lemma \ref{lem:TCL} and $d = 6a \sqrt{L}$. Throughout the proof, $\phi(L)$ will denote a function that may depend on $\delta$ but not on $\s$ such that $\underset{L \rightarrow \infty}{\lim} \phi(L) = 0$. Let $\Lambda_1,...\Lambda_r$ be the particle clusters of $\s_0$ of size $\ge \displaystyle \underline{v}\dfrac{1}{v}B(\s_0)$ (with their neighboring sites that are empty in $\s$). More precisely, the intervals $\Lambda_k = [a_k,b_k]$ are such that $b_k-a_k \ge \dfrac{\underline{v}}{v}B(\s_0)$, $\s_0(a_k)=\s_0(b_k)=0$ and $\restriction{\s_0}{(a_k,b_k)} \equiv 1$. Defining the $\Lambda_k$ defines the class $\mathcal{C}$ and we have of course $\s_0 \in \mathcal{C}$. Note already that even though $r$ now depends on $\s$, it can simply be bounded by $\dfrac{v}{\underline{v}\delta}$, which only depends on $\delta$. We aim to prove that after a time $\tau_1 = \dfrac{B(\s_0)}{2v} + 3ar\sqrt{\el(\s_0)}$, the configuration is in $\h(\Lambda,d)$. To do that, we first focus on the particle clusters.\\
We define now the iterations of $F$ starting from $\s_0$. For all $1 \le k \le r, \ \s^{(i)} := F(\s^{(i-1)})$ with $\s^{(0)}:=\s_0$. If we ever have $\s^{(i)} \notin \mathcal{C}$, then set for example $F(\s^{(i+1)}) = \s^{(i)}$ and $t(\s^{(i)}) = 0$. We will not encounter this case later on.\\
By Lemma \ref{lem:pdecroissant}, we have by induction that 
\begin{equation*}
\P\left( p(\s^{(r)}) = 0 \right) \ge 1 - r^2 \varepsilon + \phi(L).
\end{equation*}
Note that $p(\s^{(r)})=0$ simply means that $\s^{(r)} \in \bigcap_{k=1}^r \h(\Lambda_k,d)$.\\
Let $T = \sum_{i=0}^{r-1} t(\s^{(i)})$, that is the time at which we reach the configuration $\s^{(r)}$. Our goal now is to bound $T$ from above. To do that, we use Equation \eqref{eq:shrinkingclusters} and by induction we get 
\begin{equation*}
\P\left( T \le \dfrac{B(\s_0)}{2v} + 2ar\sqrt{\el(\s_0)} \right) \ge 1 - r^2\varepsilon + \phi(L).
\end{equation*}
Now using the exact same argument as in Lemma \ref{lem:goodstaysgood}, as long as $T \le \dfrac{B(\s_0)}{2v} + 2ar\sqrt{\el(\s_0)}$ and $\s_T \in  \bigcap_{k=1}^r \h(\Lambda_k,d)$, then with high probability these zeros will remain at time $\tau_1 = \dfrac{B(\s_0)}{2v} + 3ar\sqrt{\el(\s_0)}$. This leads to the result we were looking for, that at time $\tau_1$, with high probability, every initial particle cluster contains a lot of zeros:
\begin{equation}\label{eq:clustersgood}
\P\left(\s_{\tau_1} \in  \bigcap_{k=1}^r \h(\Lambda_k,d) \right) \ge 1 - r^2\varepsilon + \phi(L).
\end{equation}

We can now handle the rest of $\Lambda$. Initially, we have zeros in $\Lambda \setminus \bigcup_{k=1}^r \Lambda_k$ that are spaced at most $\dfrac{\underline{v}}{v}B(\s_0)$ apart. This threshold was precisely chosen so that we can again use the supercritical contact process to ensure that at time $t_0 = \dfrac{B(\s_0)}{2v}$, we have:
\begin{equation*}
\P_{\s_0} \left( \s_{t_0} \notin \h(\Lambda \setminus \bigcup_{k=1}^r \Lambda_k,d) \right) \le C r t_0 e^{-c\sqrt{L}} .
\end{equation*}
Again, since $t_0 \le \tau_1$, this property persists at time $t_1$. Putting this together with equation \eqref{eq:clustersgood} finally gives:
\begin{equation}\label{eq:everythinggood}
\underset{L \rightarrow \infty}{\liminf} \ \P\left(\s_{\tau_1} \in \h(\Lambda,d) \right) \ge 1 - r^2\varepsilon.
\end{equation}
Here, we used the fact that $r$ is bounded (with $L$). More precisely, we have $r \le \dfrac{v L}{\underline{v}B(\s_0)} \le \dfrac{v}{\underline{v}\delta}$.

From there, it only remains to apply Proposition \ref{relaxationfinie2} with $\el = d$ and $\beta = 1/4$ to find that for any local function $f$ such that $\mu(f) = 0$, at time $t:= \tau_1 + \dfrac{6a\sqrt{L}}{2\underline{v}} + (6a\sqrt{L})^{1/4}$,

$$\underset{L \rightarrow \infty}{\limsup} \left|\E[f(\s_t)]\right| \le r^2\varepsilon.$$

Finally, we can bound this time replacing $L$ with $B(\s_0)/\delta$ to get a consistent expression:
$$t \le \dfrac{B(\s_0)}{2v}+\dfrac{\alpha \sqrt{B(\s_0)}}{\sqrt{\delta}},$$
with $\alpha$ depending only on $\varepsilon$ and $q$.
The fact that all the generic constants $c,C$ used in the proof do not depend on $\s_0$ and that the CLT used is uniform in the initial configuration guarantees that the bound above is also uniform in $\s_0$.
\end{proof}

\subsection{Conclusion}

We can now prove Theorem \ref{mainthm}.
\begin{proof}[Proof of Theorem \ref{mainthm}]

First, we proved Equation \eqref{eq:mainthm2} with Proposition \ref{prop:cutofffull} in the previous subsection. Next, let $\delta >0$, $\s \in \Omega_L^\delta$ and $$t'_1(\s) = \left( \dfrac{B(\s)}{2\underline{v}} + (B(\s))^{1/4} \right) \vee \left( \dfrac{(\log L)^9}{2\underline{v}} + (\log L)^{9/4} \right).$$\\
If $B(\s) \le (\log L)^9$, then we use the fact than $\s \in \h(\Lambda,(\log L)^9)$ to use Proposition \ref{relaxationfinie2} with $\el = (\log L)^9$ and $\beta = 1/4$. This way, we get that for any local function $f$ with support in $\Lambda$:
$$|\E_\s[f(\s_{\frac{(\log L)^9}{2\underline{v}} + (\log L)^{9/4}})]| \le C ||f||_\infty L e^{-c(\log L)^{9/8}},$$
which goes to zero as $L \rightarrow \infty$ uniformly in $\s$.\\
In this case, $t'_1(\s) = \left( \dfrac{(\log L)^9}{2\underline{v}} + (\log L)^{9/4} \right) \le t_1(\s)$ so we get our result.\\
If $B(\s) > (\log L)^9$, then by using Proposition \ref{relaxationfinie2} with $\el=B(\s)$ and $\beta = 1/4$, we have:
$$|\E_\s[f(\s_{\frac{B(\s)}{2\underline{v}} + B(\s)^{1/4}})]| \le C ||f||_\infty L e^{-cB(\s)^{1/8}} \le C ||f||_\infty L e^{-c(\log L)^{9/8}},$$
which again leads to our result.\\

Finally, let $\Phi : \N \rightarrow \N$ such that $\Phi(L) \underset{L \rightarrow +\infty}{\longrightarrow} +\infty$. Let $\s \in \h(0,L+1,\Phi(L))^c$.
Let $\Lambda_1$ be the largest particle cluster in $\s$:\\
$\Lambda_1 = [X+1,Y-1]$ with $\s(X)=\s(Y)=0$, $\restriction{\s}{[X+1,Y-1]} \equiv 1$ and $Y-X \ge \Phi(L)$. Let $M = (X+Y)/2$.
By Lemma \ref{lem:TCL}, we see that with $t = \dfrac{B(\s)}{2v} - \dfrac{2a}{v}\sqrt{B(\s)},$
\begin{multline}\label{eq:proofeq3} \P_{\s} \left( X_t \in [M - 3a \sqrt{B(\s)}, M - a\sqrt{B(\s)}],\right. \\ \left. Y_t \in [M + a \sqrt{B(\s)}, M + 3a\sqrt{B(\s)}] \right) \ge 1 - \varepsilon + \phi(L) - Ce^{-c \sqrt{B(\s)}},
\end{multline}
with $\phi(L)$ going to $0$ when $L \rightarrow +\infty$ uniformly in $\s$. Here, $r = 1$, we only have one particle cluster.
This event implies that $\restriction{\s_t}{[M - a\sqrt{B(\s)}, M + a\sqrt{B(\s)}]} \equiv 1$. Now, note that:
\begin{equation} \label{eq:proofeq3b}
\mu(\one_{\restriction{\s_t}{[M - a\sqrt{B(\s)}, M + a\sqrt{B(\s)}]} \equiv 1}) = (1-q)^{2a\sqrt{B(\s)}} \le (1-q)^{2a \Phi(L)},
\end{equation}
which goes to zero when $L \rightarrow \infty$.\\
Putting together Equations \eqref{eq:proofeq3} and \eqref{eq:proofeq3b} concludes the proof of Equation \eqref{eq:mainthm3} and thus the proof of Theorem \ref{mainthm}.
\end{proof}

Now we can conclude the proof of the cutoff.

\begin{proof}[Proof of Theorem \ref{thm:cutoff}]
The proof of the lower bound is straightforward by taking initial configuration $\oneL$ and applying Equation \eqref{eq:mainthm3}.\\

Let $\delta = \dfrac{\underline{v}}{v}$ and $\varepsilon >0$. For every $\s \in \Omega^\delta_\Lambda$, we have that $t_1(\s) \le \dfrac{\delta L}{2\underline{v}} + (\delta L)^{1/4} = \dfrac{L}{2v} + (\delta L)^{1/4}$. If $\s \in (\Omega^\delta_\Lambda)^c$, then $t_2(\s) \le \dfrac{L}{2v} + 8a \sqrt{L}$. So in every case, for $L$ large enough, if $t = \dfrac{L}{2v} + 8a \sqrt{L}$, we have:
$$d(t) \le \varepsilon.$$
\end{proof}


\begin{appendices}
\section{Convergence behind the front}\label{sec:app}

Before proving Theorem \ref{ergo}, we recall a result that deals with the relaxation for the FA-1f process far from the front.

\begin{prop}[{\cite[Theorem 5.1]{BDT19}}]\label{relaxation}
Let $q > \bar{q}$, $\alpha < 1/2$ and $\delta > 0$. There exists $c >0$ such that for all $t \ge 0$, for any $M \le e^{\delta t^\alpha}$, any $f$ with support in $[0,M]$, such that $\mu(f)=0$ and $||f||_\infty \le 1$, for all $\s_0 \in LO$ such that $\tilde{\s}_0 \in \h(\underline{v}t,M+(4\bar{v}-\underline{v})t,\sqrt{t})$,  then 
$$| \E\left[ f(\theta_{3\bar{v}t}\tilde{\s}_t) \right] | \le e^{-c\sqrt{t}}.$$ 
\end{prop}

The same kind of result holds for a process on $\Z_-$:

\begin{prop}\label{relaxation2}
Let $q > \bar{q}$, $\alpha < 1/2$ and $\delta > 0$. There exists $c >0$ such that for any $f$ with support $[0,M]$, $\mu(f)=0$ and $||f||_\infty \le 1$, for any $t \ge 0$, $M \le e^{\delta t^\alpha}$, for all $\e_0 \in LO_-$ such that $\tilde{\e}_0 \in \h(\underline{v}t,M+(4\bar{v}-\underline{v})t,\sqrt{t})$ and $X_0 \le -M - (3\bar{v}-\underline{v})t$, then 
$$| \E\left[ f(\theta_{3\bar{v}t}\tilde{\e}_t) \right] | \le e^{-c\sqrt{t}}.$$
\end{prop}

\begin{proof}
The proof is almost identical to that of the previous proposition. However, in this case we have to make sure that the interval $[3\bar{v}t, 3\bar{v}t + M]$ seen from the front does not go out of the domain $\Z_-$. Recall that with probability $1-O(e^{-Bt})$ we have $X_t - X_0 \le -\underline{v}t$. Then as long as $-X_0 + \underline{v}t \ge 3\bar{v}t + M$, we have $X_t + 3\bar{v}t + M \le 0$ with probability $1-O(e^{-Bt})$.
\end{proof}

We now prove Theorem \ref{ergo}. The proof relies on the same coupling as in \cite{BDT19}. Let us fix $\e_0 \in LO_-$. We pick $\s_0 \in LO$ an arbitrary configuration such that $X(\e_0) = X(\s_0)$. We will now prove that there exist $d^* >0$ and $c>0$ such that 
\begin{equation} \label{coupling} ||\tilde{\mu}^{\e_0}_t - \tilde{\mu}^{\s_0}_t ||_{[0,d^*t]} \le \exp\left(-c e^{(\log t)^{1/4}}\right).
\end{equation} 

Throughout the proof, $c,C$ denote generic constants independent of $\s_0, \e_0$. These constants may vary from line to line. Before diving into the complete coupling, we shall define a coupling that comes into play during the proof: the $\Lambda-$maximal coupling. If $\Lambda \subset \Z$ is finite, and $\mu, \mu'$ are two probability measures on $\Omega_-$, we define the \emph{$\Lambda$-maximal coupling} between $\mu$ and $\mu'$ as follow:
\begin{enumerate}
\item we sample $\restriction{(\s, \s')}{\Lambda \times \Lambda}$ according to the maximal coupling, i.e. the one that achieves the total variation distance of the marginals of $\mu$ and $\mu'$ on $\Omega_\Lambda$,
\item we sample $\restriction{\s}{\Z_- \setminus \Lambda}$ and $\restriction{\s'}{\Z_- \setminus \Lambda}$ independently according to their conditional distributions $\mu(.|\restriction{\s}{\Lambda})$ and $\mu'(.|\restriction{\s'}{\Lambda})$.
\end{enumerate}

Throughout this section, $\mu$ denotes the Bernoulli product measure on $\Z_-$: $\mu = \mu^{1-q}_{\Z_-}$.

For $t >0$, let us define the following quantities:
\begin{itemize}
\item $\varepsilon = \dfrac{\bar{v}}{2(\bar{v}+\underline{v})}$,
\item $t_0 = (1-\varepsilon)t$,
\item $\Delta_1 = \exp\left((\log t)^{1/4}\right)$,
\item $\Delta_2 = (\log t)^{3/4}$,
\item $\Delta = \Delta_1 + \Delta_2.$
\end{itemize}

The time interval $[0,t]$ is divided in several sub-intervals. First, during a time $t_0 = (1-\varepsilon)t$ (which is the longest phase), we aim to use the Zeros Lemma, which provides a good amount of empty sites behind the front. We then divide the remaining time $\varepsilon t$ into $N$ steps of length $\Delta$, with $N = \lfloor \frac{\varepsilon t}{\Delta} \rfloor$ and a remainder step of length $\varepsilon t - N\Delta$. Each of these steps of length $\Delta$ consists in two steps of respective lengths $\Delta_1$ and $\Delta_2$. We define $t_n = t_0 + n\Delta$ and $s_n=t_n + \Delta_1$, $n \le N$.\\

We define the interval $\Lambda_n = [3\bar{v}\Delta_1, \underline{v}s_{n} - (\bar{v}+\underline{v})\Delta_1]$ and the distance $d_n = \underline{v}t_n - (\bar{v}+\underline{v})\Delta n$. Note that $d_n$ slightly differs from the original proof.
Lastly, let us introduce the following notation. For $(A_n)$ a sequence of measurable subsets of $LO \times LO_-$, we write:

$$P_{A_n}(\cdot) := \underset{(\tilde{\s},\tilde{\e}) \in A_n}{\sup} \ \P\left(\cdot  | \ (\tilde{\s}_{t_n},\tilde{\e}_{t_n}) = (\tilde{\s},\tilde{\e}) \right).$$

Depending on the context, $t_n$ can be replaced by $s_n$ in the definition above.\\

We are now ready to explain the coupling between $\s_t$ and $\e_t$ that will lead to Theorem \ref{coupling}.
\begin{itemize}
\item At time $t_0$, we sample $\s_{t_0}$ and $\e_{t_0}$ independently according to the laws $\mu^{\e_0}_{t_0}$ and $\mu^{\s_0}_{t_0}$.
\item Suppose the coupling at time $t_n$ constructed. If the configurations $\tilde{\s}_{t_n}$ and $\tilde{\e}_{t_n}$ coincide on the interval $[1,d_n]$, then we let the evolve according to the standard coupling until time $t$. If not, we proceed in two steps. With high probability, both configurations will have zeros behind the front. More precisely, let $Z_n$ be the event:
$$Z_n = \{ \tilde{\s}_{t_n},\tilde{\e}_{t_n} \in \h(\underline{v}\Delta_1,\underline{v}t_n,\sqrt{\Delta_1}) \ \text{and} \ (X(\s_{t_n}) - X(\s_0)), (X(\e_{t_n}) - X(\e_0)) \le -\underline{v}t_n \}.$$
Here, the event $Z_n$ differs from the original proof to ensure that the interval $[\underline{v} \Delta_1, \underline{v}t_n]$ (seen from the front) does not go out of bounds.\\ 
From both the Zeros Lemma, Lemma \ref{vitessemin} and their analogs in \cite{BDT19}, we have for $t$ large enough, 
$$\P(Z_n) \ge 1 - 2(\underline{v}t_n)^2 e^{-c\sqrt{\Delta_1}} - Ce^{-c' t_n} \ge 1 - C {t_n}^2 e^{-c\sqrt{\Delta_1}}.$$
The last term of the first line comes from the fact that we ask for the inequalities $X(\s_{t_n}) \le X(\s_0) -\underline{v}t_n $ and $X(\e_{t_n}) \le X(\e_0) -\underline{v}t_n $. Note that none of the constants $c, C$ depend on $\e_0$ or $\s_0$.\\
In the following, it will be useful to also introduce the event $$Z'_n = \{ \tilde{\s}_{t_n},\tilde{\e}_{t_n} \in \h(0,3\bar{v}\Delta_1,2\underline{v}\Delta_1)\}.$$ We can already note that $Z_n \subset Z'_n$ for $t$ large enough.

\item At time $s_{n}$, we sample $\tilde{\s}_{s_{n}}$ and $\tilde{\e}_{s_{n}}$ according to the $\Lambda_n$-maximal coupling between the laws $\tilde{\mu}^{\s_{t_{n}}}_{\Delta_1}$ and $\tilde{\mu}^{\e_{t_{n}}}_{\Delta_1}$. Let $Q_n = \{ \tilde{\s}_{s_{n}} = \tilde{\e}_{s_{n}} \, \text{on} \, \Lambda_n \}$. First, note that by definition of the maximal coupling and by the Markov property, we have, for $(\s, \e) \in LO \times LO_-$:
\begin{align*}\P(Q_n^c \ | \s_{t_{n}} = \s, \e_{t_{n}} = \e) &= ||\tilde{\mu}^{\s}_{\Delta_1} - \tilde{\mu}^{\e}_{\Delta_1}||_{\Lambda_n}\\
						&\le ||\tilde{\mu}^{\s}_{\Delta_1} - \mu||_{\Lambda_n} + ||\tilde{\mu}^{\e}_{\Delta_1} - \mu||_{\Lambda_n}.
\end{align*}
Since the interval $\Lambda_n$ is "far from the front", then these two distances are small uniformly in $\tilde{\s}_{t_n}, \tilde{\e}_{t_n}$ as long as the event $Z_n$ is satisfied. Let us justify this claim by applying Proposition \ref{relaxation2}.\\

Let $\e \in LO_-$ be an initial configuration such that $\tilde{\e} \in \h(\underline{v}\Delta_1,\underline{v}t_n,\sqrt{\Delta_1})$. Assume also that $X(\eta) \le -\underline{v}t_n$. Let us check the hypotheses of Proposition \ref{relaxation2}. Take $$M = \underline{v}s_n - (\underline{v} + \bar{v})\Delta_1 - 3 \bar{v}\Delta_1.$$ Then $M + (4\bar{v}-\underline{v})\Delta_1 = \underline{v}t_n - \underline{v}\Delta_1 \le \underline{v} t_n$, which guarantees that we do have $\tilde{\e} \in \h(\underline{v}\Delta_1, M+(4\bar{v}-\underline{v})\Delta_1,\sqrt{\Delta_1})$. Next, since $X(\e) \le - \underline{v}t_n$, we find: $$-M-(3\bar{v}-\underline{v})\Delta_1 = -\underline{v}t_n + (\bar{v}+\underline{v})\Delta_1 \ge -\underline{v}t_n \ge X(\eta).$$
Finally, we do have $M \le e^{\Delta_1^\alpha}$ for $\alpha = 1/4$ for example.
\\
We can now apply Proposition \ref{relaxation2} that gives $||\tilde{\mu}^{\e}_{\Delta_1} - \mu||_{\Lambda_n} \le C e^{-c \sqrt{\Delta_1}}$ with $C,c$ independent of $\e$. We can handle the term $||\tilde{\mu}^{\s}_{\Delta_1} - \mu||_{\Lambda_n}$ the same way in order to find in the end:

$$\label{eqQn} P_{Z_n}(Q_n^c)  \le Ce^{-c \sqrt{\Delta_1}}.$$

\item Next, if the event $Q_n$ occurs, we define $\tilde{x}$ as the distance between the front and the leftmost zero of $\tilde{\s}_{s_n}$ (or $\tilde{\e}_{s_n}$) located in $\Lambda_n$. If there is none, define $\tilde{x}$ as the distance between the front and the right boundary of $\Lambda_n$.
Let $\beta$ be a Bernoulli random variable (independent of everything else constructed so far) such that $\P(\beta = 1) = e^{-2\Delta_2}$. The event $\{\beta = 1\}$ has the same probability as the event that two independent Poisson clocks do not ring during a time $\Delta_2$. We now construct the configurations at time $t_{n+1}$.
\begin{itemize}
\item  If $\beta = 1$, we consider that, seen from the front at time $s_n$, the clocks associated with the sites $0$ and $\tilde{x}$ do not ring during $\Delta_2$. First, we sample the configuration at the left of $0$ according to the standard coupling, with empty boundary condition. Let $\xi_{n+1}$ be the common increment of the front during this time, namely $\xi_{n+1} = X_{t_{n+1}} - X_{s_n}$. Next, we fix $\tilde{\s}_{t_{n+1}}(-\xi_{n+1}) = \tilde{\e}_{t_{n+1}}(-\xi_{n+1}) := 0$ and $\tilde{\s}_{t_{n+1}}(\tilde{x}-\xi_{n+1}) = \tilde{\e}_{t_{n+1}}(\tilde{x}-\xi_{n+1}) := \tilde{\s}_{s_n}(\tilde{x})$ and sample the configurations $\tilde{\s}_{t_{n+1}}$ and $\tilde{\e}_{t_{n+1}}$ according to the maximal coupling on $[1- \xi_{n+1}, \tilde{x}-1 - \xi_{n+1}]$, with boundary conditions $0$ on $0$ and $\tilde{\s}_{s_n}(\tilde{x})$ on $\tilde{x}$. To the right of this interval, we sample the configurations according to the standard coupling with the appropriate boundary conditions. 
\item If $\beta = 0$, then we let evolve $(\tilde{\s}_{s_n},\tilde{\e}_{s_n})$ for a time $\Delta_2$ via the standard coupling conditioned to have at least one ring either at $\tilde{x}$ or at $0$.
\end{itemize} 

First, let us notice that $\tilde{x}$ is likely to be not too far right inside $\Lambda_n$. Indeed, under the event $Z_n$ the law of the configuration in $\Lambda_n$ is close to the product measure at time $s_n$. Let $B_n := \{\tilde{x} \le 3\bar{v}\Delta_1 + \frac{\sqrt{\Delta_2}}{2}\}$. Then, by Proposition \ref{relaxation}:
\begin{align}\label{eqBn}
\mathcal{P}_{Z_n}\left(B_n^c \cap Q_n \right) &\le \mathcal{P}_{Z_n}\left(\restriction{(\tilde{\s}_{s_{n}})}{[3\bar{v}\Delta_1,3\bar{v}\Delta_1 + \frac{\sqrt{\Delta_2}}{2}]} \equiv 1\right)\\
	&\le p^{\frac{\sqrt{\Delta_2}}{2}} + O(e^{-c\sqrt{\Delta_1}}).
\end{align}

Now, denote by $Z''_n$ the event $\left\{\tilde{\s}_{s_n}, \tilde{\e}_{s_n} \in \h\left(\frac{\sqrt{\Delta_2}}{2}, 3\bar{v}\Delta_1,\frac{\sqrt{\Delta_2}}{2}\right)\right\}$, and recall that $Z_n \subset Z'_n$. Using the second case of Zeros Lemma \ref{zeros1} and its equivalent \cite[Lemma 4.4]{BDT19} (with $L = \sqrt{\Delta_2}/2, \ L+M = 3 \bar{v}\Delta_1, \ \el=L$), we find that: 
\begin{align}\label{EqZ}
P_{Z_n}((Z_n'')^c) &\le \frac{2 \Delta_1^2}{\sqrt{\Delta_2}}\exp(-c\sqrt{\Delta_2}) + \Delta_1 \left(3\bar{v}\Delta_1 - \frac{\sqrt{\Delta_2}}{2}\right) \exp(-c \Delta_2) \\
			&\le O(e^{-c \sqrt{\Delta_2}}).
\end{align}

With the event $Z''_n$, we are now able the use Proposition \ref{relaxationfinie} to couple the configuration at time $t_n$ on the interval $J_n := [1 - \xi_{n+1},\tilde{x} - \xi_{n+1} - 1]$.\\
First, note that $Z''_n \cap B_n \subset \{ \tilde{\s}_{s_n}, \tilde{\e}_{s_n} \in \h \left (0, \tilde{x},\sqrt{\Delta_2} \right) \}$. Let $L = 3\bar{v}\Delta_1 + \frac{\sqrt{\Delta_2}}{2} \le e^{\Delta_2^{1/4}}$. Then by Proposition \ref{relaxationfinie}, we find that for any $y \in [3 \bar{v} \Delta_1,3 \bar{v}\Delta_1 + \frac{\sqrt{\Delta_2}}{2}]$:
\begin{equation}
\P\left(\tilde{\s}_{t_{n+1}} \ne \tilde{\e}_{t_{n+1}} \ \text{on} \ [1 - \xi_{n+1},y - \xi_{n+1} - 1] \ | \ \tilde{x} = y \ \text{and} \ \tilde{\s}_{s_n}  \ \tilde{\e}_{s_n} \in \h (0, y,\sqrt{\Delta_2}) \right)  = O(e^{-c\sqrt{\Delta_2}}).
\end{equation}
Again, this bound does not depend on $y$ or the initial configurations. This in turn shows that:
\begin{equation}\label{equalityJn}
\P_{Z_n'', B_n}(\tilde{\s}_{t_{n+1}} \ne \tilde{\e}_{t_{n+1}} \ \text{on} \ J_n) = O(e^{-c\sqrt{\Delta_2}}).
\end{equation}
 Under the event $Q_n$, the configurations seen from the front at time $s_n$ coincide on the interval $\Lambda_n$. On the right part of this interval, namely  $[\tilde{x}, \underline{v}s_n - (\bar{v}+\underline{v})\Delta_1]$, we let the configurations evolve according to the standard coupling. The matching of the configurations will then persist on the sub-interval  $[\tilde{x}-\xi_{n+1}, \underline{v}s_n - (\bar{v}+\underline{v})\Delta_1 - \xi_{n+1} - \bar{v}\Delta_2]$  w.h.p. thanks to Proposition \ref{eventF}.
Note that $\underline{v}s_n - (\bar{v}+\underline{v})\Delta_1 - \bar{v}\Delta_2 = \underline{v}t_{n+1} - \bar{v}\Delta \ge d_{n+1}$. Denote by $K_n$ the interval $[\tilde{x} - \xi_{n+1}, d_{n+1}]$.
We found that:
\begin{equation}\label{equalityKn}
P_{Q_n}\left(\tilde{\s}_{t_{n+1}} \ne \tilde{\e}_{t_{n+1}} \ \text{on} \ K_n \ | \  \beta = 1 \right) \le \P(F(0,\bar{v}\Delta_2,\Delta_2) +  \P(\xi_{n+1} \ge 0)  = O(e^{-c \Delta_2}).
\end{equation}

Finally, notice that if the clock attached to the site $0$ (i.e the front at time $s_n$) does not ring, and we let the configurations at the left of this site evolve according to the standard coupling, we naturally get:
\begin{equation}\label{equalityleft}
P_{Q_n}\left(\tilde{\s}_{t_{n+1}} \ne \tilde{\e}_{t_{n+1}} \ \text{on} \ [0, -\xi_{n+1}] \ | \ \beta = 1 \right) = 0.
\end{equation}

From Equations \eqref{equalityJn}, \eqref{equalityKn} and \eqref{equalityleft}, we can conclude that:

\begin{equation}\label{eq:coupling}
P_{Q_n, Z''_n, B_n}\left( \tilde{\s}_{t_{n+1}} \ne \tilde{\e}_{t_{n+1}} \ \text{on} \ [1,d_{n+1}] \ | \ \beta = 1 \right) = O(e^{-c\sqrt{\Delta_2}}).
\end{equation}
\end{itemize}

We are now ready to estimate the probability that the coupling "succeeds". Let $M_n = \{\tilde{\s}_{t_n} = \tilde{\e}_{t_n} \ \text{on} \ [1,d_n] \}$. At time $t_{n+1}$, depending on the coupling at time $t_n$, we can write:
\begin{align}
\P(M_{n+1}^c) &\le \P(M_{n+1}^c\cap M_n^c \cap Z_n) + \P(Z_n^c) + \P(M_{n+1}^c\cap M_n) \nonumber \\
				&\le \P(M_{n+1}^c | M_n^c \cap Z_n)\P(M_n^c) + \P(Z_n^c) + \P(M_{n+1}^c\cap M_n). \label{eqrec}
\end{align}

Let us focus on the term $\P(M_{n+1}^c | M_n^c \cap Z_n)$. We condition on the events that can happen at time $s_n$:
$$\P(M_{n+1}^c | M_n^c \cap Z_n) \le P_{Z_n, M_n^c}(M_{n+1}^c | Z_n'' \cap Q_n \cap B_n) + P_{Z_n}((Z_n'')^c) + P_{Z_n}(Q_n^c) + P_{Z_n}(B_n^c \cap Q_n).$$

Out of all these terms, only the first one is (a priori) not vanishing when $t$ goes to infinity according to equations \eqref{eqBn}, \eqref{EqZ}, \eqref{eqQn}. Thanks to equation \eqref{eq:coupling}, for $t$ large enough:
\begin{align}
P_{Z_n, M_n^c}(  M_{n+1}^c  | Z_n'' \cap Q_n \cap B_n) &\le 1 - \P(\beta = 1) \left(1 - P_{M_n^c \cap Z_n} \left(M_{n+1}^c | \{\beta = 1 \} \cap Z_n'' \cap Q_n \cap B_n \right) \right) \\
 & \le 1 - \frac{1}{2} e^{-2 \Delta_2}.
\end{align}

Back to the initial inequality \eqref{eqrec}, we can plug our estimate for the first term, and easily bound the remaining two terms to find:

\begin{equation}
\P(M_{n+1}^c) \le \left(1 - \frac{1}{2}e^{-2\Delta_2}\right) \P(M_n^c) + Ce^{-c\Delta}
\end{equation}

It only remains to solve this recursive equation with the values of $N,\Delta_1,\Delta_2$ we chose. We find:
\begin{align}
\P(M_N^c) &\le \left(1 - \frac{1}{2}e^{-2\Delta_2} \right)^N + C 2^{-N+1} e^{-c\Delta - 2(N-1)\Delta_2}\\
		&\le O(e^{-c\Delta_1}).
\end{align}

Let us now conclude the proof. At time $t_N$, with probability $\P(M_N)$ the configurations $\tilde{\s}_{t_N}$ and $\tilde{\e}_{t_N}$ match on the interval $[1,d_N]$. From time $t_N$ to time $t$, we let the configurations evolve according to the standard coupling. In the same way as we did before, we can estimate the interval on which the configurations are equal at time $t$ given the event $M_N$.

\begin{align}
\P(\tilde{\s}_t = \tilde{\e}_t \ \text{on} \ [1,d^*t]) &= \P(\tilde{\s}_t = \tilde{\e}_t \ \text{on} \ [1,d^*t] \  | \ M_N) \P(M_N) \\
				&\ge \P( \{X_t \le X_{t_N}\} \cap F(d_N, d^*t, t - t_N)^c) \P(M_N).
\end{align}
It only remains to chose a $d^*$ such that the event $F(d_N, d^*t, t - t_N)$ has low probability. Let $d^* = \frac{\underline{v}}{4}$. Then using $t - t_N \le \Delta$ and replacing $\varepsilon$ and $d^*$ by their values, we find:
\begin{equation}  d_N - d^*t \ge \frac{\underline{v}t}{4} - \underline{v}t \ge \bar{v}\Delta \ge \bar{v}(t - t_N),
\end{equation}
which implies $\P(F(d_N, d^*t, t - t_N)) = O(e^{-ct})$. The event $\{X_t \le X_{t_N}\}$ has probability $1 - O(e^{-c \Delta})$ so we find in the end the announced estimate:
\begin{equation}\label{eqfinale} \P(\tilde{\s}_t \ne \tilde{\e}_t \ \text{on} \ [1,d^*t]) = O(e^{-c\Delta}).
\end{equation} 
The existence of an invariant measure $\nu$ comes from the compacity of the set of probability measures on $\Omega$. The uniqueness and the convergence estimate of Theorem \ref{ergo} come from Equation \eqref{eqfinale}. \qed

\end{appendices}

\textbf{Acknowledgments.} The author would like to thank Oriane Blondel for her many useful advice and Fabio Toninelli for his careful reading. This project has been supported by the ANR grant LSD (ANR-15-CE40-0020).

\bibliographystyle{plain}
\bibliography{bibliographie}

\end{document}

%% file: dessin_lemma0.tex

\begin{tikzpicture}[scale=0.7]


\draw[-, very thick] (4.15,0)--(5.35,0); 
\draw[-, very thick] (10.15,0)--(10.45,0);
\draw[-, very thick] (15.15,0)--(17.35,0);

\draw[<->] (3,0.35)--(6.5,0.35) node[above,midway,scale=1.5]{$\Lambda_{1}$}; 
\draw[<->] (9,0.35)--(11.5,0.35) node[above,midway,scale=1.5]{$\Lambda_{2}$}; 
\draw[<->] (14,0.35)--(18.5,0.35) node[above,midway,scale=1.5]{$\Lambda_{3}$}; 

\draw[dashed] (0,0)--(4.15,0);
\draw[dashed] (5.35,0)--(10.15,0);
\draw[dashed] (10.45,0)--(15.15,0);
\draw[dashed] (17.35,0)--(20,0);

\draw[-] (3,0.3)--(3,-0.3);
\draw[-] (6.5,0.3)--(6.5,-0.3);
\draw[-] (9,0.3)--(9,-0.3);
\draw[-] (11.5,0.3)--(11.5,-0.3);
\draw[-] (14,0.3)--(14,-0.3);
\draw[-] (18.5,0.3)--(18.5,-0.3);


\draw (20.5,0) node[right,xshift=0cm, scale=1.5]{$\sigma \in \mathcal{C}$};


\draw[line width=8pt,dashed,opacity=0.4] (3,0)--(4.15,0);
\draw[line width=8pt,dashed,opacity=0.4] (5.35,0)--(6.5,0);
\draw[line width=8pt,dashed,opacity=0.4] (9,0)--(10.15,0);
\draw[line width=8pt,dashed,opacity=0.4] (10.45,0)--(11.5,0);
\draw[line width=8pt,dashed,opacity=0.4] (14,0)--(15.15,0);
\draw[line width=8pt,dashed,opacity=0.4] (17.35,0)--(18.5,0);


 \draw[thick] (4.15,0) circle(0.1) node[below,yshift=-0.1cm,scale=1]{$X^1$}; 
 \fill[white] (4.15,0) circle(0.1);
 
 \draw[thick] (5.35,0) circle(0.1) node[below,yshift=-0.1cm,scale=1]{$Y^1$}; 
 \fill[white] (5.35,0) circle(0.1);
 
 \draw[thick] (10.15,0) circle(0.1);
 \fill[white] (10.15,0) circle(0.1);
 
 \draw[thick] (10.45,0) circle(0.1);
 \fill[white] (10.45,0) circle(0.1); 
 
 \draw[thick] (15.15,0) circle(0.1) node[below,yshift=-0.1cm,scale=1]{$X^3$};
 \fill[white] (15.15,0) circle(0.1); 
 
 \draw[thick] (17.35,0) circle(0.1) node[below,yshift=-0.1cm,scale=1]{$Y^3$};
 \fill[white] (17.35,0) circle(0.1);
 
\draw[<->] (10.15,-0.5)--(10.45,-0.5) node[below,midway]{$\le d$};


\draw[-] (0,0.3)--(0,-0.3);
\draw[-] (20,0.3)--(20,-0.3);

 \draw[thick] (-0.25,0) circle(0.2); 
 \fill[white] (-0.25,0) circle(0.2);

 \draw[thick] (20.25,0) circle(0.2); 
 \fill[white] (20.25,0) circle(0.2);

\end{tikzpicture}

%% file: dessin_lemma1.tex

\begin{tikzpicture}[scale=0.7]
\draw[-, very thick] (8.7,0)--(10.90,0);
\draw[-, very thick] (3,3)--(17,3) node[midway,above,yshift=0.05cm,scale=1.2]{$M^{k_0}$}; 

\draw[dashed] (0,0)--(8.7,0);
\draw[dashed] (0,3)--(3,3);

\draw[dashed] (10.9,0)--(20,0);
\draw[dashed] (17,3)--(20,3);


\draw (20.5,3) node[right,xshift=-0.5cm, scale=1.5]{$\sigma$};
\draw (20.5,0) node[right,xshift=-0.4cm, scale=1.5]{$\sigma_{t(\sigma)}$};

\draw[dotted] (10,3)--(10,0);


\draw[-] (8.35,0.3)--(8.35,-0.3);
\draw[-] (9.45,0.3)--(9.45,-0.3);
\draw[-] (10.55,0.3)--(10.55,-0.3);
\draw[-] (11.65,0.3)--(11.65,-0.3);

\draw[<->] (8.35,-0.4)--(11.65,-0.4) node[below, midway]{$6a\sqrt{\ell (\sigma)}$};
\draw[->, thick] (3,-0.3)--(8.35,-0.3) node[below, midway]{$\sim vt(\sigma)$};


 \draw[-] (3,3) to[out=-5,in=170] (8.7,0); 
 \draw[-] (17,3) to[out=190,in=10] (10.9,0);


 \draw[thick] (3,3) circle(0.2) node[above,yshift=0.1cm,scale=1.3]{$X^{k_0}$}; 
 \fill[white] (3,3) circle(0.2);
 
 \draw[thick] (8.7,0) circle(0.2)node[above,yshift=0.2cm,scale=1.3]{$X^{k_0}_{t(\sigma)}$};
 \fill[white] (8.7,0) circle(0.2);
 
 \draw[thick] (17,3) circle(0.2) node[above,yshift=0.1cm,scale=1.3]{$Y^{k_0}$};
 \fill[white] (17,3) circle(0.2);
 
 \draw[thick] (10.9,0) circle(0.2)node[above,yshift=0.2cm,scale=1.3]{$Y^{k_0}_{t(\sigma)}$};
 \fill[white] (10.9,0) circle(0.2);

\end{tikzpicture}

%% file: dessin_lemma2.tex

\begin{tikzpicture}[scale=0.7]


\draw[-, very thick] (3,3)--(6.5,3);
\draw[-, very thick] (9,3)--(11.5,3);
\draw[-, very thick] (14,3)--(18.5,3);

\draw[dashed] (0,3)--(3,3);
\draw[dashed] (6.5,3)--(9,3);
\draw[dashed] (11.5,3)--(14,3);
\draw[dashed] (18.5,3)--(20,3);


\draw[-, very thick] (4.15,0)--(5.35,0); 
\draw[-, very thick] (10.15,0)--(10.45,0); 
\draw[-, very thick] (15.15,0)--(17.35,0);

\draw[dashed] (0,0)--(4.15,0);
\draw[dashed] (5.35,0)--(10.15,0);
\draw[dashed] (10.45,0)--(15.15,0);
\draw[dashed] (17.35,0)--(20,0);


\draw (20.5,3) node[right,xshift=0cm, scale=1.5]{$\sigma$};
\draw (20.5,0) node[right,xshift=0cm, scale=1.5]{$\sigma_{t(\sigma)}$};


\draw[dotted] (3,3)--(3,0);
\draw[dotted] (6.5,3)--(6.5,0);
\draw[dotted] (9,3)--(9,0);
\draw[dotted] (11.5,3)--(11.5,0);
\draw[dotted] (14,3)--(14,0);
\draw[dotted] (18.5,3)--(18.5,0);

\draw[line width=8pt,dashed,opacity=0.4] (3,0)--(4.15,0);
\draw[line width=8pt,dashed,opacity=0.4] (5.35,0)--(6.5,0);
\draw[line width=8pt,dashed,opacity=0.4] (9,0)--(10.15,0);
\draw[line width=8pt,dashed,opacity=0.4] (10.45,0)--(11.5,0);
\draw[line width=8pt,dashed,opacity=0.4] (14,0)--(15.15,0);
\draw[line width=8pt,dashed,opacity=0.4] (17.35,0)--(18.5,0);


 \draw[-] (3,3) to[out=-65,in=130] (4.15,0); 
 \draw[-] (6.5,3) to[out=240,in=45] (5.35,0);
 
 \draw[-] (9,3) to[out=-65,in=130] (10.15,0); 
 \draw[-] (11.5,3) to[out=240,in=45] (10.45,0);
 
 \draw[-] (14,3) to[out=-65,in=130] (15.15,0); 
 \draw[-] (18.5,3) to[out=240,in=45] (17.35,0);


 \draw[thick] (3,3) circle(0.2) node[above,yshift=0.1cm,scale=1.3]{$X^{1}$}; 
 \fill[white] (3,3) circle(0.2);
 
 \draw[thick] (6.5,3) circle(0.2) node[above,yshift=0.1cm,scale=1.3]{$Y^{1}$}; 
 \fill[white] (6.5,3) circle(0.2);
 
 \draw[thick] (9,3) circle(0.2)node[above,yshift=0.1cm,scale=1.3]{$X^2$};
 \fill[white] (9,3) circle(0.2);
 
 \draw[thick] (11.5,3) circle(0.2) node[above,yshift=0.1cm,scale=1.3]{$Y^2$};
 \fill[white] (11.5,3) circle(0.2); 
 
 \draw[thick] (14,3) circle(0.2) node[above,yshift=0.1cm,scale=1.3]{$X^{3}$};
 \fill[white] (14,3) circle(0.2); 
 
 \draw[thick] (18.5,3) circle(0.2) node[above,yshift=0.1cm,scale=1.3]{$Y^{3}$};
 \fill[white] (18.5,3) circle(0.2);
 

 \draw[thick] (4.15,0) circle(0.1) node[below,yshift=-0.1cm,scale=1]{$X^{1}_{t(\sigma)}$}; 
 \fill[white] (4.15,0) circle(0.1);
 
 \draw[thick] (5.35,0) circle(0.1) node[below,yshift=-0.1cm,scale=1]{$Y^{1}_{t(\sigma)}$}; 
 \fill[white] (5.35,0) circle(0.1);
 
 \draw[thick] (10.15,0) circle(0.1);
 \fill[white] (10.15,0) circle(0.1);
 
 \draw[thick] (10.45,0) circle(0.1);
 \fill[white] (10.45,0) circle(0.1); 
 
 \draw[thick] (15.15,0) circle(0.1) node[below,yshift=-0.1cm,scale=1]{$X^{3}_{t(\sigma)}$};
 \fill[white] (15.15,0) circle(0.1); 
 
 \draw[thick] (17.35,0) circle(0.1) node[below,yshift=-0.1cm,scale=1]{$Y^{3}_{t(\sigma)}$};
 \fill[white] (17.35,0) circle(0.1);
 
\draw[<->] (10.15,-0.5)--(10.45,-0.5) node[below,midway]{$\le 6a\sqrt{L}$};


\draw[-] (0,0.3)--(0,-0.3);
\draw[-] (0,3.3)--(0,2.7);
\draw[-] (20,0.3)--(20,-0.3);
\draw[-] (20,3.3)--(20,2.7);

 \draw[thick] (-0.25,0) circle(0.2); 
 \fill[white] (-0.25,0) circle(0.2);

 \draw[thick] (-0.25,3) circle(0.2); 
 \fill[white] (-0.25,3) circle(0.2);
  
 \draw[thick] (20.25,0) circle(0.2); 
 \fill[white] (20.25,0) circle(0.2);
 
 \draw[thick] (20.25,3) circle(0.2); 
 \fill[white] (20.25,3) circle(0.2); 

\end{tikzpicture}

%% file: dessin_lemma3.tex

\begin{tikzpicture}[scale=0.7]


\draw[-, very thick] (3.5,3)--(5.35,3);
\draw[-, very thick] (14.5,3)--(18.2,3);

\draw[dashed] (0,3)--(4.35,3);
\draw[dashed] (4.65,3)--(14.5,3);
\draw[dashed] (18.5,3)--(20,3);


\draw[-, very thick] (4.35,0)--(4.65,0); 
\draw[-, very thick] (15.15,0)--(17.35,0);

\draw[dashed] (0,0)--(4.35,0);
\draw[dashed] (4.65,0)--(15.15,0);
\draw[dashed] (17.35,0)--(20,0);


\draw (20.5,3) node[right,xshift=0cm, scale=1.5]{$\sigma$};
\draw (20.5,0) node[right,xshift=0cm, scale=1.5]{$\sigma_{t(\sigma)}$};

 
\draw[<->] (3,4)--(6.5,4) node[above,midway,scale=1.5]{$\Lambda_1$};  
\draw[<->] (9,4)--(11.5,4) node[above,midway,scale=1.5]{$\Lambda_2$}; 
\draw[<->] (13.5,4)--(19,4) node[above,midway,scale=1.5]{$\Lambda_3$}; 
 
\draw[-] (3,0.3)--(3,-0.3);
\draw[-] (3,3.3)--(3,2.7);
\draw[-] (6.5,0.3)--(6.5,-0.3);
\draw[-] (6.5,3.3)--(6.5,2.7);

\draw[-] (9,0.3)--(9,-0.3);
\draw[-] (9,3.3)--(9,2.7);
\draw[-] (11.5,0.3)--(11.5,-0.3);
\draw[-] (11.5,3.3)--(11.5,2.7);

\draw[-] (13.5,0.3)--(13.5,-0.3);
\draw[-] (13.5,3.3)--(13.5,2.7);
\draw[-] (19,0.3)--(19,-0.3);
\draw[-] (19,3.3)--(19,2.7);


\draw[line width=8pt,dashed,opacity=0.4] (3,3)--(3.5,3);
\draw[line width=8pt,dashed,opacity=0.4] (5.35,3)--(6.5,3);

\draw[line width=8pt,dashed,opacity=0.4] (9,3)--(11.5,3);

\draw[line width=8pt,dashed,opacity=0.4] (13.5,3)--(14.5,3);
\draw[line width=8pt,dashed,opacity=0.4] (18.2,3)--(19,3);


\draw[dotted] (3,3)--(3,0);
\draw[dotted] (6.5,3)--(6.5,0);
\draw[dotted] (9,3)--(9,0);
\draw[dotted] (11.5,3)--(11.5,0);
\draw[dotted] (13.5,3)--(13.5,0);
\draw[dotted] (19,3)--(19,0);

\draw[line width=8pt,dashed,opacity=0.4] (3,0)--(4.35,0);
\draw[line width=8pt,dashed,opacity=0.4] (4.65,0)--(6.5,0);

\draw[line width=8pt,dashed,opacity=0.4] (9,0)--(11.5,0);

\draw[line width=8pt,dashed,opacity=0.4] (13.5,0)--(15.35,0);
\draw[line width=8pt,dashed,opacity=0.4] (17.45,0)--(19,0);


 \draw[-] (3.5,3) to[out=-65,in=130] (4.35,0); 
 \draw[-] (5.5,3) to[out=240,in=45] (4.65,0);
 
 \draw[-] (14.5,3) to[out=-65,in=130] (15.35,0); 
 \draw[-] (18.3,3) to[out=240,in=45] (17.45,0);


 \draw[thick] (3.5,3) circle(0.2) node[above,yshift=0.1cm,scale=1]{$X^{1}$}; 
 \fill[white] (3.5,3) circle(0.2);
 
 \draw[thick] (5.5,3) circle(0.2) node[above,yshift=0.1cm,scale=1]{$Y^{1}$}; 
 \fill[white] (5.5,3) circle(0.2);
 
 \draw[thick] (14.5,3) circle(0.2) node[above,yshift=0.1cm,scale=1]{$X^{3}$};
 \fill[white] (14.5,3) circle(0.2); 
 
 \draw[thick] (18.3,3) circle(0.2) node[above,yshift=0.1cm,scale=1]{$Y^{3}$};
 \fill[white] (18.3,3) circle(0.2);
 

 \draw[thick] (4.35,0) circle(0.1); 
 \fill[white] (4.35,0) circle(0.1);
 
 \draw[thick] (4.65,0) circle(0.1); 
 \fill[white] (4.65,0) circle(0.1);
 
 \draw[thick] (15.35,0) circle(0.1) node[below,yshift=-0.1cm,scale=1]{$X^{3}_{t(\sigma)}$};
 \fill[white] (15.35,0) circle(0.1); 
 
 \draw[thick] (17.45,0) circle(0.1) node[below,yshift=-0.1cm,scale=1]{$Y^{3}_{t(\sigma)}$};
 \fill[white] (17.45,0) circle(0.1);
 
\draw[<->] (4.35,-0.5)--(4.65,-0.5) node[below,midway]{$\le d$};


\draw[-] (0,0.3)--(0,-0.3);
\draw[-] (0,3.3)--(0,2.7);
\draw[-] (20,0.3)--(20,-0.3);
\draw[-] (20,3.3)--(20,2.7);

 \draw[thick] (-0.25,0) circle(0.2); 
 \fill[white] (-0.25,0) circle(0.2);

 \draw[thick] (-0.25,3) circle(0.2); 
 \fill[white] (-0.25,3) circle(0.2);
  
 \draw[thick] (20.25,0) circle(0.2); 
 \fill[white] (20.25,0) circle(0.2);
 
 \draw[thick] (20.25,3) circle(0.2); 
 \fill[white] (20.25,3) circle(0.2); 

\end{tikzpicture}